\documentclass[11pt, reqno]{amsart}

\usepackage{amsmath,amsfonts,amsthm,amssymb,color,hyperref}
\usepackage[normalem]{ulem}
\usepackage{mathrsfs,mathtools,soul}

\usepackage[T1]{fontenc}
\usepackage{tikz} 
\usetikzlibrary{arrows,decorations.pathmorphing,backgrounds,fit,positioning,shapes.symbols,chains}
\usepackage{graphicx}
\usepackage{subfigure} 
\usepackage[left=1.1in, right=1.1in, top=1.1in,bottom=1.1in]{geometry}
\usepackage{graphicx}

\usepackage{bbm}
\usepackage[textsize=small]{todonotes}
\usepackage{todonotes}
\usepackage{xcolor}
\definecolor{boo}{rgb}{1.0, 0.0, 1.0}

\newcommand{\D}{\mathbb D}
\newcommand{\R}{\mathbb R}

\newcommand{\E}{\mathbb E}
\def\i{\mathbf{i}}
\def\wt{\widetilde}

 \newcommand{\1}{{\bf 1}}

\def\H{\mathfrak H}
\newcommand{\F}{\mathscr F}
\newcommand{\e}{\varepsilon}
\newcommand{\Var}{ {\rm Var}}
\newcommand{\Cov}{{\rm Cov}}
\newcommand{\ep}{\varepsilon}

\newtheorem{theorem}{Theorem}[section]
\newtheorem{lemma}[theorem]{Lemma}
\newtheorem{proposition}[theorem]{Proposition}

\theoremstyle{definition}
\newtheorem{definition}[theorem]{Definition}
\theoremstyle{remark}
\newtheorem{remark}{Remark}
\numberwithin{equation}{section}
\begin{document}
\title[Spatial averages for the PAM driven by rough noise]{Spatial averages for the Parabolic Anderson model driven by rough noise}

\author[D. Nualart]{David Nualart$^{\ast,1}$} \thanks{David Nualart is supported by NSF Grant DMS 1811181.}
  
\author[X. Song]{Xiaoming Song$^{\ddagger,2}$}
  
\author[G. Zheng]{Guangqu Zheng$^{\ast,3}$}

 \maketitle
 \vspace{-0.5cm}

\begin{center}  {\small  University of Kansas$^\ast$ and Drexel University$^\ddagger$; 

\smallskip

Emails: nualart@ku.edu$^{1}$, xs73@drexel.edu$^2$,  zhengguangqu@gmail.com$^3$ } \end{center}

\begin{abstract} 
In this paper, we study  spatial averages for the  parabolic Anderson model in the Skorohod sense driven by  rough Gaussian noise, which is colored  in space and time. We include the case of a fractional noise  with Hurst parameters
$H_0$ in time and $H_1$ in space, satisfying  $H_0 \in (1/2,1)$, $H_1\in (0,1/2)$ and $H_0 + H_1 > 3/4$.
Our main result is a functional central limit theorem for the spatial averages.  
As an important ingredient of our analysis, we present a  Feynman-Kac formula that is new  for these values of the Hurst parameters.
\end{abstract}



\allowdisplaybreaks

\section{Introduction}
 We consider the   Parabolic Anderson Model (PAM) on $\R_+\times \R$
\begin{align} \label{PAM1}
 \begin{cases}
  \dfrac{\partial u }{\partial t}  = \dfrac{1}{2}   \Delta u + u  \diamond \dot{W}, \\
  \quad\\
  u(0,\cdot)\equiv 1,  \end{cases}
 \end{align}
  where $\dot{W}(t,x)$ is a generalized centered Gaussian random field  with    covariance
\begin{equation} \label{cov}
\E\big[  \dot{W}(t,x) \dot{W}(s,y)\big] = \gamma_0(t-s) \gamma( x-y).
 \end{equation}
 The product in \eqref{PAM1}  is the Wick product and the mild solution is defined in the  Skorohod sense.
We assume  that  the covariance  (\ref{cov}) satisfies one of the following two (overlapping) sets of  conditions:

\smallskip

{\bf (H1)} $\gamma_0:\R \rightarrow \R_+$  is a nonnegative-definite  locally integrable function and  $\gamma$ is a tempered distribution, whose Fourier transform   $\mu$  admits a density $\varphi$ that satisfies the following \textit{modified   Dalang's condition}:
\begin{align} \tag{D}\label{mDc}
\int_{\R}\frac{\varphi(x)^2}{1 + x^2} dx < +\infty. 
\end{align}
 We also assume that $\varphi$ is {\it continuous at zero} with $\varphi(0) = 0$ and the following {\it concavity condition} is satisfied: $\exists \ \kappa_0\in (0,\infty)$ such that
\begin{align}\tag{C}
\varphi(x+y) \leq \kappa_0\big[  \varphi(x) + \varphi(y) \big]  ~ \text{for every $x,y\in\R$}. \label{concond}
\end{align}

{\bf (H2)} $\gamma_0(t) = |t |^{2H_0 - 2}$ for some $H_0 \in( 1/2,1)$ and $\varphi(z) =  | z |^{1-2H_1}$   for some $H_1\in(0,1/2)$ such that\footnote{In the fractional   case
we actually have $\varphi(z) = c(H_1) | z |^{1-2H_1}$, with $c(H_1)= \frac {\sin(\pi H_1)}{2\pi} \Gamma(2H_1+1)$, but to simplify the presentation we take $c(H_1)=1$.} $H_0 + H_1 > 3/4$.

The existence of a unique mild solution under condition {\bf (H1)} has been established in \cite{HLN16,HLN17}. Notice, that unlike many of the works on  the stochastic heat equation with spatial colored noise,  in {\bf (H1)}  the correlation function  $\gamma$  does not need to be a measure. Previously the random field solution theory for the stochastic heat equation  driven by a Gaussian noise white in time and rough  in space,  was restricted to the range $H_1\in(1/4, 1/2)$ (see, for instance,  \cite{BJQ,HHLNT1,HHLNT2}).
Condition $H_0 + H_1 > 3/4$ in  case {\bf (H2)}  breaks the barrier $1/4$, which  is required under {\bf (H1)} if
$\varphi(x)=|x|^{1-2H_1}$, and it has been observed in \cite{SSX19} that the random field solution exists  for the PAM when $H_0\in(1/2,1)$, $H_1\in (0,1/2)$  and $H_0 + H_1 > 3/4$.  

We are interested in deriving a functional  central limit theorem (CLT) for spatial averages of the form
\begin{align}\label{ATR}
A_t(R):=\int_{-R}^R \big[ u(t,x) - 1 \big] dx,
 \end{align}
 as $R \to \infty$, where $t\in[0,\infty)$.
 Using the chaos expansion of the solution (see  \eqref{soln}) and a 
chaotic CLT (see Theorem \ref{thm00})), we prove the following main result.

   \begin{theorem} \label{MainResult} Let $A_t(R)$ be defined as in \eqref{ATR} and suppose that one of the assumptions {\rm \textbf{(H1)} or \textbf{(H2)}} holds. Then, as $R\to \infty$,
 \begin{center}
$\left\{ \dfrac{1}{\sqrt{R}}A_t(R): t\in\R_+ \right\}$ converges in law to a centered Gaussian process $\mathcal{G}$ on $C(\R_+;\R)$,
 \end{center}
where for any $t_1, t_2\in\R_+$,
\begin{align*}
\E\big[ \mathcal{G}_{t_1}  \mathcal{G}_{t_2}  \big] = 2 \int_{\R}     \E\Big[ \mathfrak{g}\big(  \mathcal{I}^{1,2}_{t_1, t_2}(z)  \big)    \Big]    dz,
\end{align*}
with $\mathfrak{g}(z) = e^z - z -1$ and  $\mathcal{I}^{1,2}_{t_1, t_2}(z) $ being the random variable  defined in Proposition {\rm \ref{prop1}}.
 \end{theorem}

Since the work \cite{HNV18} of Huang, Nualart and Viitasaari, there have been a series of papers devoted to central limit theorems  for spatial averages of  stochastic partial differential equations. 
The authors of  \cite{HNV18} consider the  nonlinear stochastic heat equation driven by space-time white noise on $\R_+\times\R$ and they are able to provide a functional central limit theorem for  $A_t(R)$.  Their methodology begins with rewriting $A_t(R)=\delta(V_{t,R})$ as a Skorohod integral and  then appeal to the recent Malliavin-Stein approach \cite{bluebook} to obtain a quantitative CLT for $A_t(R)$ in total variation distance. The convergence of finite-dimensional distributions can be proved  using the multivariate Malliavin-Stein bound and the tightness is established  using Kolmogorov's criterion. Other key tools used in \cite{HNV18} are  Clark-Ocone  formula, Burkholder's inequality and they essentially rely on the assumption that the underlying Gaussian noise is white in time so as to render us a martingale structure. Soon later, the authors of  \cite{HNVZ19} consider the same equation with spatial dimension $d\geq 1$; while imposing that the Gaussian noise is white in time and  it has a spatial covariance given  by the Riesz kernel,  they establish a functional CLT and a quantitative CLT for  spatial averages. We also refer  interested readers to several other investigations on stochastic heat equations in   \cite{CKNP19, CKNP20, GL20, KNP20, PU20} and on stochastic wave equations in \cite{BNZ20, DNZ18, NZ20a}; these papers more or less follow the path paved by the work \cite{HNV18}, although the nature of the problems and the techniques differ.

In the above references, the underlying  Gaussian noise is always white in time, which is crucial to apply stochastic calculus  techniques.
These techniques are not available  in our framework and  we will make use of Wiener chaos expansions and Feynman-Kac formulas.
In \cite{NZ19BM}, that is  the closest work to the present paper, Nualart and Zheng consider the PAM \eqref{PAM1} on $\R_+\times\R^d$ with the correlation kernels  $\gamma_0$ and $\gamma$ satisfying the following conditions:
\begin{itemize}
\item[(i)]
$\gamma_0:\R\to \R_+$ is nonnegative-definite and  locally integrable.

\item[(ii)]  $\gamma$ is  a positive finite  measure, expressed as the Fourier transform of some nonnegative tempered measure $\mu$ that satisfies   Dalang's condition (\cite{Dalang99, HHNT15})
$\int_{\R^d} \frac{\mu(d\xi)}{1 + \| \xi\|^2}<\infty$, where $\| \cdot \|$ denotes the Euclidean norm.

\end{itemize}
   In \cite{NZ19BM},  the Gaussian fluctuation is established for 
\[
\int_{\{\| x\| \leq R\}} [u(t,x) -1 ]dx
\]
for each $t>0$ and, under the extra integrability condition on $\gamma_0$
\begin{equation}  \label{Int}
\int_0^{t_0}\int_0^{t_0} \gamma_0(s-t) s^{-\alpha} t^{-\alpha} dsdt <\infty,  \quad  \text{for some} \,\, \alpha\in (0,1/2), \,\, t_0>0,
\end{equation}
 the functional CLT holds as well; see \cite[Theorems 1.6, 1.9]{NZ19BM} for more precise statements.   Note also that in our Theorem \ref{MainResult}, we do not  need to assume  condition \eqref{Int} used  to guarantee the tightness.

Observe that unlike in the papers \cite{HNVZ19, DNZ18}, the variance order of $A_t(R)$ is $R$, which does not depend on the parameters of the covariance, for example,  the Hurst index $H_1$ in the setting \textbf{(H2)}. This is due to our assumption $\varphi(0)=0$ in both settings of  \textbf{(H1)} and \textbf{(H2)}, that forces the negligibility of the first  chaotic component of $A_t(R)$ in the limit, while the other chaoses contribute to the order $R$.   This situation is completely different from the case $H_1 >1/2$, where a nonchaotic behavior occurs.

 In what follows, we   briefly sketch the main steps of the proof of Theorem \ref{MainResult}.
The first step in proving Theorem \ref{MainResult} is  to show the order of the limiting variance and more generally, to establish the limit covariance structure.

\begin{proposition}\label{PROP:COV1}  Assume hypothesis  $\mathbf{(H1)}$ or $\mathbf{(H2)}$.  For $t_1, t_2\in\R_+$,
\begin{equation}\label{eqn-1}
\frac{1}{2R}  \Cov\big( A_{t_1} (R), A_{t_2}(R) \big) \xrightarrow{R\to\infty} \int_\R dz   \E\Big[   e^{ \mathcal{I}^{1,2}_{t_1, t_2}(z)} -\mathcal{I}^{1,2}_{t_1, t_2}(z) -1  \Big],
\end{equation}
where  $\mathcal{I}^{1,2}_{t_1, t_2}(z)$ is  defined in Proposition \ref{prop1}. Moreover if  $\Pi_p(\bullet)$ denotes the orthogonal projection   onto the $p$th Wiener chaos associated to $W$ $($see Section \ref{SEC2} for more details$)$, we have
  for $p=1$, 
\begin{equation}\label{eqn-2}
\lim\limits_{R\to\infty}\frac{1}{2R}\Cov\big( \Pi_1 A_{t_1} (R), \Pi_1 A_{t_2}(R) \big) = 0,
\end{equation}
 and for each $p\geq 2$,
\begin{equation}\label{eqn-3}
\lim\limits_{R\to\infty}\frac{1}{2R}  \Cov\big( \Pi_p A_{t_1} (R), \Pi_p A_{t_2}(R) \big)= \frac{1}{p!} \int_\R dz  \E\Big[  \left(\mathcal{I}^{1,2}_{t_1, t_2}(z)\right)^p  \Big].
\end{equation}
\end{proposition}

 An important  ingredient to prove   Proposition \ref{PROP:COV1} is  the following  Feynman-Kac formula for  the  moments  of the solution.  
 In the case {\bf (H1)},   this formula   is essentially a reformulation of  Corollary 4.4 in \cite{HLN17} and the result for  {\bf (H2)}  is new, so we provide  in Section \ref{FKproof}  a unified proof for both cases. 

\begin{proposition}\label{prop1} In both cases {\bf (H1)} and  {\bf (H2)}, we have  for any $(s_i, t_i,x_i, y_i)\in (0,\infty)^2\times\R^2$, $i=1,\dots, k$ $(k\geq 2)$
\begin{align}
&\E\left[\prod_{i=1}^k u(t_i, x_i)\right] =  \E\left[  \exp\left(  \sum_{1\leq i<j \leq k} \mathcal{I}^{i,j}_{t_i, t_j}(x_i-x_j) \right) \right],   \label{prop:FK} 
\end{align}
where
\[
 \mathcal{I}^{i,j}_{t, s}(z) :=\int_0^{t}\int_0^{s} drdv\gamma_0(r-v) \int_{\R}d\xi  \varphi(\xi) e^{-\i (B^i_{t-r} - B_{s-v}^j + z) \xi} 
\]
is understood as the $L^p(\Omega)$-limit  $($for any $p\geq 1$$)$ of
\[
\int_0^{t}\int_0^{s} drdv\gamma_0(r-v) \int_{\R}d\xi e^{-\e \xi^2}  \varphi(\xi) e^{-\i (B^i_{t-r} - B_{s-v}^j + z) \xi} =:  \mathcal{I}^{i,j}_{t, s,\e}(z),
\]
with  $B^1,\dots, B^k$  being i.i.d. standard   Brownian motions on $\R$.  
Moreover, for each $i<j$ and for any $\lambda>0$,
\[
\sup\left\{ \E\Big[ \exp\big( \lambda \vert  \mathcal{I}^{i,j}_{t, s,\e}(z) \vert \big)\Big] : \e>0,~ z\in\R \right\}<\infty.
\]

\end{proposition}

Once Proposition \ref{PROP:COV1} is proved, we  apply the  multivariate chaotic central limit theorem  to establish the convergence in law of finite-dimensional distributions (f.d.d.). This chaotic CLT 
is a consequence of the well-known  fourth moment theorems \cite{bluebook,NP05,PT05}.

\begin{theorem}[Multivariate chaotic CLT]\label{thm00}  Fix an integer $n\geq 1$ and consider a family $\big\{ A_R: R> 0 \big\}$ of centered  random vectors in $\R^n$ such that each component of $A_R=(A_{R, 1}, \ldots, A_{R,n})$ belongs to $ L^2(\Omega, \sigma\{W\}, \mathbb{P})$ and has the following chaos expansion
\[
\qquad\qquad\qquad A_{R,j} = \sum_{q=1}^\infty I_q^W (  g_{q, j,R}  )\quad \text{with $g_{q, j,R}\in \H^{\odot q}$},
\]
where $I^W_q$ denotes the $q$th-multiple stochastic integral with respect to $W$ $($see Section 2.2$)$.
Suppose the following conditions {\rm (a)-(d)} hold:
\begin{itemize}
\item[(a)]  $\forall i,j\in\{1, \ldots, n\}$ and  $\forall q\geq 1$, $ \E\big[ I_q^W (  g_{q, j,R}  ) I_q^W (  g_{q, i,R}  )\big]\xrightarrow{R\to+\infty}  \sigma_{i,j,q}$.
\medskip
\item[(b)]  $\forall i\in\{1, \ldots, n\}$, ${\displaystyle \sum_{q=1}^\infty \sigma_{i,i,q} < \infty}$.
\medskip
\item[(c)]   For any $1\leq r \leq q-1$,     
$
\big\| g_{q, i,R}\otimes_{r} g_{q, i,R} \big\|_{ \H^{\otimes (2q-2r)}} \xrightarrow{R\to +\infty} 0$.
 \medskip
\item[(d)]  $\forall i\in\{1, \ldots, n\}$, ${\displaystyle \lim_{N\to+\infty} \sup_{R>0} \sum_{q=N+1}^\infty   \E\big[   I_q^W (  g_{q, i,R}  )^2 \big]  = 0}$.

\end{itemize}
Then $A_R$ converges in law to $N(0, \Sigma)$ as $R\to+\infty$, where $\Sigma = \big( \sigma_{i,j} \big)_{i,j=1}^n$ is given by  $\sigma_{i,j} = \sum_{q=1}^\infty \sigma_{i,j,q}$.
\end{theorem}

We refer to \cite{CNN, bluebook, HN05} for more details  on this result and to  Section  2.2 for the definition of  
the Hilbert space $\H$ and the $r$-contraction $\otimes_r$ in our setting.
 As the last step in the proof of Theorem \ref{MainResult}, we will establish the tightness  by using Kolmogorov's criterion. The key tool is the hypercontractivity property of the Ornstein-Uhlenbeck generator, see \eqref{HYPER}.

\begin{remark}
Theorem \ref{MainResult} and Proposition \ref{prop1} also hold under Hypothesis {\bf (H1)} if $\gamma_0=\delta_0$, that means, if the noise is white in time. The proofs are   similar and  we omit the details. In this framework, the random variables 
$\mathcal{I}^{i,j}_{t, s,\e}(z)$ defined in Proposition \ref{prop1} would have the expression
\[
\int_0^{t\wedge s}  dr \int_{\R}d\xi e^{-\e \xi^2}  \varphi(\xi) e^{-\i (B^i_{t-r} - B_{s-r}^j + z) \xi} .
\]
\end{remark}

The rest of this article is organized as follows: Section \ref{SEC2} is devoted to preliminary knowledge that is required for later proofs, and  we prove Proposition \ref{PROP:COV1} in Section \ref{SEC3}; we show the f.d.d. convergence and tightness respectively in Section \ref{SEC4} and Section \ref{SEC5}; the last section contains the proof of  Feynman-Kac formula.

\section{Preliminaries}\label{SEC2} 

We first introduce some handy  notation here.

\subsection{Notation} For $r\in \mathbb{N}$ and $\pmb{x_r}=(x_1,\dots, x_r)$, we write   $d\pmb{x_r}=dx_1\dots dx_r$, $\mu(d\pmb{x_r})=\mu(dx_1)\dots\mu(dx_r)$ and $\tau(\pmb{x_r})=x_1+\dots+x_r$. For integers $1\leq r<p$, we write $(\xi_1,\dots, \xi_p)=\pmb{\xi_p}=(\pmb{\xi_r}, \pmb{\eta_{p-r}})$ with $ \pmb{\xi_r}=(\xi_1, \dots, \xi_r)$ and $ \pmb{\eta_{p-r}}=(\xi_{r+1}, \dots, \xi_p)$.  

 For any $p\in\mathbb{N}$, $\mathfrak{S}_{p}$ denotes the set of permutations of $\{1, 2, \dots, p\}$, and for any $\pmb{s_p}=(s_1, \dots, s_p)$ we write   $\pmb{s^\sigma_p}=(s_{\sigma(1)}, \dots, s_{\sigma(p)})$.

 For any interval $I$, we use $|I|$ to denote its length.

For any $t>0$ and $m\in \mathbb{N}$,  put $\Delta_m(t):= \big\{ \pmb{r_m} \in\R_+^m: t>r_1> \cdots> r_m>0  \big\} $ and $\text{SIM}_m(t) : = \big\{ \pmb{w_m} \in\R_+^m: w_1+\cdots +w_m\leq t  \big\} $.

For any $p>0$ and any random variable $X\in L^p(\Omega)$, we write $\Vert X\Vert_p=\big(\E[ |X|^p]\big)^{1/p}$.
 
\subsection{Wiener chaos expansion} 
For $\phi, \psi \in C_c^\infty(\R_+\times\R)$, we define 
\begin{align*}
\langle \phi,\psi \rangle_\H  = \int_{\R_+^2}dsdt \gamma_0(s-t)  \int_{\R} d\xi \varphi(\xi) \F \phi(s, \xi) \F \psi(t, -\xi) 
\end{align*}
where 
$$
 \F \phi(s, \xi)  : = \int_{\R} e^{-\i x \xi} \phi(s, x) dx
 $$ 
is the Fourier transform with respect to the spatial variable only. Due to our assumptions ({\bf (H1)} or {\bf (H2)}), the above functional  $(\phi, \psi)\longmapsto\langle \phi,\psi \rangle_{\H}$ defines an inner product, under which $C_c^\infty(\R_+\times\R)$ can be extended to a Hilbert space, denoted by $\H$.  

We can view the noise $W$ as an isonormal Gaussian process over $\H$, that is, $\{W(h):  h\in\H\}$ is a centered Gaussian family with  
\[
\E\big[ W(\phi) W(\psi)  \big] = \langle \phi,\psi \rangle_\H,
\]
for any $\phi, \psi\in\H$. For any $n\in \mathbb{N}$, we denote by $\H^{\otimes n}$ the $n$th tensor product of $\H$ and by $\H^{\odot n}$ the symmetric subspace of $\H^{\otimes n}$. 

It is a well-known fact that $L^2(\Omega, \sigma\{W\}, \mathbb{P})$ can be decomposed into an infinite  orthogonal sum:
\[
L^2(\Omega, \sigma\{W\}, \mathbb{P}) = \bigoplus_{p= 0}^\infty \mathbb{C}_p,
\]
where $\mathbb{C}_p$ is called the $p$th Wiener chaos and  it is the $L^2(\Omega)$-completion of the set
\[
\big\{ H_p\big(W(\psi) \big): \| \psi\|_\H =1 \big\},
\]
with $H_p(x) = (-1)^p e^{x^2/2} \frac{d^p}{dx^p} e^{-x^2/2} $ denoting the $p$th Hermite polynomial.  For any integer $p\ge 1$, the multiple integral $I_p^W$ of order $p$ is a bounded linear operator from $\mathfrak{H}^{\otimes p}$ onto $\mathbb{C}_p$ uniquely characterized by the following conditions: 
\begin{itemize}
\item[(i)] Given any orthogonal   unit vectors $e_1, \dots  ,e_k\in\H$ $(k\geq 2)$ and any nonnegative integers $n_1, \dots, n_k$ such that $n_1 + \cdots + n_k = p$, it holds that
\[
I_p^W\Big(  e_1^{\otimes n_1} \otimes e_2^{\otimes n_2} \otimes \cdots \otimes e_k^{\otimes n_k}   \Big) = \prod_{i=1}^k H_{n_i}\big(W(e_i) \big).
\]

\item[(ii)] For any $f\in \mathfrak{H}^{\otimes p}$, $I_p^W(f) = I_p^W(\wt{f}) $, where $\wt{f} \in\mathfrak{H}^{\odot p}$ denotes the symmetrization of $f$.
\end{itemize}
For $p=0$, $\mathbb{C}_0 = \R$ and $I_0$ is the identity. 
 The following  isometry  property holds for any $f\in \mathfrak{H}^{\otimes p}$ and $g\in \mathfrak{H}^{\otimes q}$ (see \cite{bluebook, Nualart} for more details):
\[
\E\big[ I_p^W(f)  I_q^W(g) \big] =  p! \big\langle \wt{f}, \wt{g} \big\rangle_{\H^{\otimes p} } \mathbf{1}_{\{p=q \}}.
\] 
  
 Another important property of Wiener chaos is the following consequence of \emph{hypercontractivity} (see \emph{e.g.} \cite[Corollary 2.8.14]{bluebook}): If  $F\in \mathbb{C}_p$ for $p\geq 1$, then for any $k\geq 2$, 
 \begin{align}\label{HYPER}
 \big\| F \big\|_k \leq (k-1)^{p/2} \big\| F \big\|_2.
 \end{align}
 
 Let us introduce the contractions appearing in condition (c) of Theorem \ref{thm00}.
   For integers $p,q\geq 1$, the $r$-contraction   $f\otimes_r g$ for $f\in\H^{\otimes p}, g\in\H^{\otimes q}$ and $1\leq r\leq p\wedge q$ is the element of $\H^{\otimes (p+q-2r)}$ defined by
\[
 f\otimes_r g=   \sum_{i_1, \dots, i_r =1} ^\infty  \langle f, e_{i_1} \otimes \cdots  \otimes e_{i_r} \rangle_{\H^{\otimes r}}  \langle g, e_{i_1} \otimes \cdots \otimes e_{i_r} \rangle_{\H^{\otimes r}} ,
 \]
 where $\{e_i,     i\ge 1\}$ is a complete orthonormal system in $\H$.
In the particular case where $f$ and $g$ are locally integrable functions, the contraction   $f\otimes_r g$ has the following expression
\begin{align*}
& (f\otimes_r g)\big( \pmb{t_{p-r}}, \pmb{\xi_{p-r}} ;  \pmb{s_{q-r}},  \pmb{\eta_{q-r}} \big) \\
&= \int_{\R_+^{2r}\times\R^{r}}  (\F_rf)\big( \pmb{t_{p-r}}, \pmb{\xi_{p-r}}, \pmb{a_r}, \pmb{\zeta_r} \big) (\F_rg)\big( \pmb{s_{q-r}}, \pmb{\eta_{q-r}}, \pmb{b_r}, -\pmb{\zeta_r} \big) \prod_{j=1}^r \gamma_0(a_j-b_j) \varphi(\zeta_j)da_jdb_jd\zeta_j,
\end{align*}
where $\F_r$ denotes the Fourier transform with respect to the $r$ space variables. We refer readers to  the appendix of \cite{bluebook} for more explanations on the contractions.

\subsection{Malliavin calculus}
 We will denote by $D$ the derivative operator in the sense of
Malliavin calculus. That is, if $F$ is a smooth and cylindrical
random variable of the form
\begin{equation*}
F=f(W(h_1),\dots,W(h_n))\,,
\end{equation*}
with $h_i \in \H$ and  $f \in C^{\infty}_c (\R^n)$, then $DF$ is the
$\H$-valued random variable defined by
\begin{equation*}
DF=\sum_{j=1}^n\frac{\partial f}{\partial
x_j}(W(h_1),\dots,W(h_n))h_j\,.
\end{equation*}
The Sobolev space $\mathbb{D}^{1,2}$ is
the closure of the space of smooth and cylindrical random variables
under the norm
\[
\|DF\|_{1,2}=\sqrt{ \E [F^2]+\E[\|DF\|^2_{\H}]}\,.
\]
We denote by $\delta$ the adjoint of the derivative operator given
by the duality formula
\begin{equation*}
\E ( \delta (u)F ) =\E ( \langle DF,u
\rangle_{\H}  ) , 
\end{equation*}
for any $F \in \mathbb{D}^{1,2}$ and any  $u \in L^2(\Omega;
\H)$ in the domain of $\delta$. The operator $\delta$ is
also called the {\it Skorohod integral} because in the case of the
Brownian motion, it coincides with an extension of the It\^o
integral introduced by Skorohod. We refer to \cite{Nualart}
for a detailed account of the Malliavin calculus with respect to a
Gaussian process.

If $F\in \mathbb{D}^{1,2}$ and $h$ is a deterministic element of $\H$, then $Fh$ is Skorohod
integrable and, by definition, the
Wick product  $F\diamond W(h)$ equals the Skorohod integral of $Fh$, that is, 
\begin{equation*}
\delta (Fh)=F\diamond W(h).  
\end{equation*}

\subsection{Mild solution} 
 Now we define 
\[
\mathcal F_t = \sigma\big\{W(\phi) : \phi\in C^\infty_c(\R_+\times\R) \text{  has support contained in } [0,t]\times \R\big\}\vee \mathcal N,
\]
where $\mathcal N$ denotes the collection of null sets.  This gives us a filtration $\mathbb{F}:=\{\mathcal{F}_t: t\geq 0 \}$.

For each $t>0$ and $x\in \R$, $G(t,x)=(2\pi t)^{-1/2}   e^{-x^2/(2t)}$ denotes the  fundamental solution of the heat equation.

\begin{definition}
An $\mathbb{F}$-adapted random field $u=\{u(t,x): t\ge 0, x\in\R\}$ is a mild 
  solution to
\eqref{PAM1}   if for all $(t,x)\in \R_+\times\R$,  
the process $\{ G(t-s, x-y) u(s,y) \mathbf{1}_{[0,t]}(s): (s,y) \in \R_+\times \R\}$ is Skorohod integrable and    following integral equation holds
\begin{equation}\label{sko-sol}
u(t,x)=1+\int_0^t\int_{\R} G(t-s, x-y) u(s,y) W(ds,dy), 
\end{equation}
where the stochastic integral   is understood in the Skorohod sense.
\end{definition}

From the results of \cite{HLN16,HLN17,NZ19BM}, under  hypotheses {\bf (H1)} or {\bf (H2)}, there exists a unique mild solution to equation \eqref{PAM1},
 which has the following Wiener chaos expansion 
 \begin{align}\label{soln}
 u(t,x) = 1 + \sum_{n= 1}^\infty I_n^W(f_{t,x,n}).
 \end{align}
where  $ f_{t,x,n} \in\H^{\odot n}$  is given by 
\begin{align*} 
 f_{t,x,n}(\pmb{s_n}, \pmb{y_n}) = \frac 1 {n!} G(t-s_{\sigma(1)}, x - y_{\sigma(1)})   \prod_{i=1}^{n-1} G(s_{\sigma(i)}-s_{\sigma(i+1)},  y_{\sigma(i) }-y_{\sigma(i+1) }   ),
 \end{align*}
 with $\sigma\in\mathfrak{S}_{n}$ such that $t > s_{\sigma(1)} > \ldots > s_{\sigma(n)} > 0$; we refer readers to \cite{HHNT15, HLN17, SSX19} for the rigorous derivation of \eqref{soln}.

  Then the core object in this paper $A_t(R)$ has the following Wiener chaos expansion
 \[
 A_t(R) = \sum_{p=1}^\infty I_p^W\left(  \int_{-R}^R  f_{t,x,p}dx \right),
 \]
 as a consequence of stochastic Fubini.

In the end of this section, let us state a useful embedding result that is a consequence of the Hardy-Littlewood inequality. 

\subsection{Embedding of $ L^{1/H_0}(\R_+^n)$ into $\mathcal{H}^{\otimes n}$}  Let $\mathcal{H}$ be the Hilbert space associated to the fractional Brownian motion with Hurst parameter $H_0\in(1/2, 1)$. That means, $\mathcal{H}$ is the closure of  $C^\infty_c(\R_+)$ under the seminorm
\[
\langle f, g \rangle_{\mathcal{H}} = \int_{\R_+} f(s) g(t) | s-t| ^{2H_0-2} dsdt.
\]
We have the continuous embedding
$  L^{1/H_0}(\R_+) \hookrightarrow  \mathcal{H} $. More precisely, for any
    $f,g:\R_+\to\R$ with $\| f \| _{L^{1/H_0}(\R_+) } <\infty$ and $\| g \| _{L^{1/H_0}(\R_+) } <\infty$, we have
   \[
    | \langle f,g \rangle _{\mathcal{H}}| \leq C_{H_0} \| f \| _{L^{1/H}(\R_+)}  \| g \| _{L^{1/H}(\R_+)},
   \]
   where the constant $C_{H_0}$  only depends on $H_0$ (see, for instance, \cite{MMV01}). Then, by iteration, one can easily show that given two functions
   $f,g \in   L^{1/H_0}(\R^n_+)$, then
      \begin{align}  \notag
    | \langle f,g \rangle _{\mathcal{H}^{\otimes n}} |  & =  \left| \int_{\R_+^{2n}}ds_1 \cdots ds_n    dt_1 \cdots dt_n  f(s_1,\dots, s_n)g(t_1,\dots , t_n)  \prod_{j=1}^n | s_j - t_j |^{2H_0-2} \right| \\  \label{F12}
    & \le C_{H_0}^n  \| f \|_{L^{1/H_0}(\R_+^n) }    \| g \|_{L^{1/H_0}(\R_+^n) }.
    \end{align}

      \medskip
 
 \section{Limiting covariance structure} \label{SEC3}

 For any nonnegative measurable function $\Theta$ and for any integer $p\ge 1$, we define the following two quantities:
   \begin{align}
 K_{1,p}(\Theta,t) &= \int_{[0,t]^p} d\pmb{s_p}    \int_{\R^{p}}  \mu(d\pmb{\xi_p})  \Theta( \tau(\pmb{\xi_p}))   \exp\left[   -\frac 12 \Var  \sum_{j=1}^p   B_{s_j}   \xi_j     \right]     \label{K1} \\
 K_{2,p}(\Theta,t)&= \int_{[0,t]^p} d\pmb{s_p}   \left(  \int_{\R^{p}}  \mu(d\pmb{\xi_p})  \Theta( \tau(\pmb{\xi_p}))   \exp\left[   - \Var  \sum_{j=1}^p   B_{s_j}   \xi_j      \right]    \right)^{\frac 1 {2H_0}},  \label{K2}
 \end{align}
 where,  here and along the paper, $B$ denotes a standard real-valued Brownian motion.  
 
 Notice that integrating on the simplex $\Delta_p(t)$ and making the change of variables $\eta_j = \xi_1 + \cdots +\xi_j$ and
 $s_{j-1}-s_j=w_j$, $j=1,\dots, p$, with the convention $\eta_0=0$ and $s_0=0$, we obtain
  \begin{align} 
 K_{1,p}(\Theta,t)&= p!\int_{\text{SIM}_p(t)} d\pmb{w_p}    \int_{\R^{p}}  d\pmb{\eta_p} \Theta(  \eta_p)    e^{-\frac 12\sum_{k=1}^p w_k \eta _k^2}
 \prod_{j=1}^p \varphi(\eta_j-\eta_{j-1})    \label{F2} \\
K_{2,p}(\Theta,t)&= p!\int_{\text{SIM}_p(t)} d\pmb{w_p}    \left( \int_{\R^{p}}  d\pmb{\eta_p} \Theta(  \eta_p)    e^{-\sum_{k=1}^p w_k \eta_k^2}
 \prod_{j=1}^p \varphi(\eta_j-\eta_{j-1})  \right)^{\frac 1 {2H_0}}. \label{F7}
  \end{align}
 
 The next technical lemma will play a fundamental role along the paper.
  \begin{lemma}\label{LEM:H12} Under hypothesis   $\mathbf{(H1)}$, we have, for any  $N>0$,
  \begin{equation} \label{F4}
   K_{1,p}(\Theta,t) \le p! t \left(  \int_{\R} \Theta (x) (1+ \varphi(x))dx\right)  \frac 12 (8\kappa_0C_N)^{p-1} \exp\left(  \frac{tD_N }{ 2C_N}   \right),
  \end{equation}
  where the constants   $C_N$ and $D_N$, given by 
  \begin{align}\label{quan:CDN1}
C_N:= \int_{\{| \eta | \geq N\}} \frac{1+\varphi(\eta) + \varphi(\eta)^2}{\eta^2}d\eta \quad  {\rm and } \quad D_N: =  \int_{\{| \eta | \leq N\}} (1+\varphi(\eta) + \varphi(\eta)^2)d\eta,
\end{align}
are finite for each $N>0$ due to the modified Dalang's condition\footnote{ Condition (D) implies $\int_{\R} \frac { \varphi(x)+ \varphi(x)^2}{1+ x^2}dx <\infty$.} \eqref{mDc}.   Moreover 
\begin{equation} \label{F1}
K_{1,p}(1,t) \le \frac {p!}{2\kappa_0} (8\kappa_0C_N)^{p} \exp\left(  \frac{tD_N }{ 2C_N}   \right).
\end{equation}
Under hypotheses  $\mathbf{(H2)}$, we have
\begin{equation} \label{F6}
   K_{2,p}(\Theta,t) \le p!  \left(  \int_{\R} \Theta (x) (1+ |x|^{1-2H_1})dx\right)  \frac{C_1^p  \, \left(t^{p\mathfrak{h} +  \frac{1-H_1}{2H_0}} \vee t^{p\mathfrak{h} + \frac{1}{4H_0} } \right)}{\Gamma(p\mathfrak{h} +  \frac{1-H_1}{2H_0}  )\wedge\Gamma( p\mathfrak{h} + \frac{1}{4H_0}  ) },
  \end{equation}
 where $ \mathfrak{h}=\frac{2H_0+H_1-1}{2H_0}>\frac 14$; moreover,
 \begin{equation} \label{F8}
 K_{2,p} (1,t) \le p! C_2^p \frac{ t^{p \mathfrak{h}}}  {\Gamma (p \mathfrak{h} +1)},
 \end{equation}
 where $C_1$ and $C_2$ are  constants  depending only on $H_0$ and  $H_1$.
  \end{lemma}
 
 \begin{remark}
 Notice that inequalities \eqref{F1} and \eqref{F8}  \emph{cannot} be obtained from \eqref{F4} and \eqref{F6}, respectively,   by  simply putting $\Theta =1$.
 \end{remark}

\begin{proof}[Proof of Lemma \ref{LEM:H12}]
Let us first show the inequality \eqref{F1} under hypothesis $\mathbf{(H1)}$.  We will make use of the expression \eqref{F2}, where $\Theta =1$.
 Using the concavity condition \eqref{concond} and the fact that $\varphi(-x)=\varphi(x)$, we can write 
\begin{equation} \label{ECU1}
 \prod_{k=1}^p \varphi(\eta_k - \eta_{k-1}) \leq \kappa_0^{p-1} \sum_{\pmb{\beta}\in\mathcal{A}_p} \prod_{k=1}^p \varphi(\eta_k)^{\beta_k},
 \end{equation}
where $\mathcal{A}_p$ is a collection of the indices   $\pmb{\beta}=(\beta_1, \dots, \beta_p)\in  \{ 0,1,2\}^p$  satisfying 
\[
x_1 \prod_{j=2}^p (x_j + x_{j-1}) = \sum_{\pmb{\beta}\in\mathcal{A}_p} \prod_{k=1}^p  x_k^{\beta_k}.
\]
It is easy to see that the cardinality of $\mathcal{A}_p$ is $2^{p-1}$. Plugging the inequality   (\ref{ECU1}) into  (\ref{F2}) yields
\[
K_{1,p}(1,t) \le 
p! \kappa_0^{p-1} \sum_{\pmb{\beta}\in\mathcal{A}_p}    \int_{\text{SIM}_p(t)} d\pmb{w_p}    \int_{\R^{p}}  d\pmb{\eta_p}   e^{-\frac 12\sum_{k=1}^p w_k \eta_k^2}
 \prod_{j=1}^p \varphi(\eta_j)^{\beta_j}.
 \]
Following the same arguments as in the proof of Lemma 3.3 in \cite{HHNT15}, we have 
 \begin{equation}
K_{1,p}(1,t) \le      p!(2 \kappa_0)^{p-1} \sum_{k=0}^{p} \binom{p}{k} \frac{t^k}{k!} (D_N)^k (2C_N)^{p-k}, \label{reexp:1}
\end{equation}
where the quantities    $C_N$ and $D_N$ are defined in \eqref{quan:CDN1}.
   The sum in the right-hand side of \eqref{reexp:1} can be estimated as follows:
\begin{equation} \label{F3}
 \sum_{k=0}^{p} \binom{p}{k} \frac{t^k}{k!} (D_N)^k (2C_N)^{p-k} =  (2C_N)^{p} \sum_{k=0}^{p} \binom{p}{k} \frac{(tD_N)^k  }{ (2C_N)^kk!}   \leq  (4C_N)^{p} e^{  \frac{tD_N }{ 2C_N}   }.
 \end{equation}
 Substituting \eqref{F3} into \eqref{reexp:1} yields  \eqref{F1}.
 
 To show \eqref{F4}, we write, using the expression \eqref{F2} and the estimates $\varphi(\eta_p -\eta_{p-1}) \le \kappa_0 [ \varphi(\eta_p) + \varphi(\eta_{p-1})]$ and $\exp(-\frac 12w_p \eta_p^2) \le 1$,
 \begin{equation} \label{F5}
 K_{1,p}(\Theta,t) =  p  t \kappa_0  \left( \int_\R  \Theta(  \eta_p)  \varphi(\eta_p) d\eta_p  \right)  K_{1,p-1}(1,t)
 +   p  t \kappa_0  \left( \int_\R  \Theta(  \eta_p)    d\eta_p  \right)   \widetilde{K}_{1,p-1}(1,t),
 \end{equation}
 where
\[
    \widetilde{K}_{1,p-1}(1,t)=  (p-1)!
 \int_{\text{SIM}_{p-1}(t)} d\pmb{w_{p-1}}   \int_{\R^{p-1}}  d\pmb{\eta_{p-1}}     e^{-\frac 12\sum_{k=1}^{p-1} w_k \eta _k^2}
\left( \prod_{j=1}^{p-1} \varphi(\eta_j- \eta_{j-1}) \right) \varphi(\eta_{p-1}).
\]
 It is easy to show that the estimate \eqref{F1} still holds if we replace $K_{1,p}(1,t)$ by $\widetilde{K}_{1,p}(1,t)$. 
 Therefore, substituting the estimate
 \[
 \max(K_{1,p-1}(1,t), \widetilde{K}_{1,p-1}(1,t) ) \le  \frac {   (p-1)!}{2\kappa_0} (8\kappa_0C_N)^{p-1} \exp\left(  \frac{tD_N }{ 2C_N}   \right)
 \]
 into  \eqref{F5}, we obtain \eqref{F4}.

Now let us  prove \eqref{F6}  under hypothesis {\bf (H2)}. We observe that, using the expression \eqref{F7}, we can write
\begin{align*}
K_{2,p} (\Theta,t) &=  p!  \int_{\text{SIM}_p(t)} d\pmb{w_p} \left( \int_{\R}d\eta_p  \Theta(\eta_p)    \int_{\R^{p-1}}  d\pmb{\eta_{p-1}}    e^{-\sum_{k=1}^p w_k \eta_k^2}   \prod_{j=1}^p \varphi(\eta_j - \eta_{j-1} ) \right)^{\frac{1}{2H_0}}  \\
 \leq&\ p!   \int_{\text{SIM}_p(t)} d\pmb{w_p} \left( \sum_{\pmb{\beta}\in\mathcal{A}_p}\int_{\R}d\eta_p  \Theta( \eta_p)    \int_{\R^{p-1}}  d\pmb{\eta_{p-1}}    e^{-\sum_{k=1}^p w_k \eta_k^2}  \prod_{j=1}^p|\eta_j  |^{\beta_j(1-2H_1)} \right)^{\frac{1}{2H_0}}   \\
 \leq&\ p!  \sum_{\pmb{\beta}\in\mathcal{A}_p} \int_{\text{SIM}_p(t)} d\pmb{w_p} \left( \int_{\R}d\eta_p  \Theta( \eta_p)    \int_{\R^{p-1}}  d\pmb{\eta_{p-1}}    e^{-\sum_{k=1}^p w_k\eta_k^2}  \prod_{j=1}^p |\eta_j  |^{\beta_j(1-2H_1)} \right)^{\frac{1}{2H_0}} . 
  \end{align*}
For $\pmb{\beta}\in\mathcal{A}_p$, we have $\beta_p\in\{0,1\}$.  As a consequence,
\begin{align}
  K_{2,p} (\Theta,t) 
 &   \le C_\Theta   t p! \sum_{\pmb{\beta}\in\mathcal{A}_p} \int_{\text{SIM}_{p-1}(t)} d\pmb{w_{p-1}}  \prod_{j=1}^{p-1} \left(     \int_{\R}  d\eta_j    e^{- w_j \eta_j^2} |\eta_j  |^{\beta_j(1-2H_1)} \right)^{\frac{1}{2H_0}} \notag\\
 &=   C_\Theta    t p! \sum_{\pmb{\beta}\in\mathcal{A}_p}  \prod_{j=1}^pK\big( \beta_j(1-2H_1)\big)^{\frac{1}{2H_0}}  \int_{\text{SIM}_{p-1}(t)} d\pmb{w_{p-1}}  \prod_{j=1}^{p-1} w_j^{-\frac{1+ \beta_j(1-2H_1)}{4H_0} } , \notag 
   \end{align}
  where   $C_\Theta = \int_{\R} \Theta(x) (1+ |x| ^{1-2H_1} ) dx$ and we use the notation
\begin{align*} 
K(\theta) := \int_{\R} e^{-x^2} x^{\theta}dx, ~\theta \ge 0.
\end{align*} 
 It is clear that by our assumptions on $(H_0, H_1)$,  that the quantities $\alpha_j:= -\frac{1+ \beta_j(1-2H_1)}{4H_0}$ belong to the interval  $(-1,0)$, so that 
  \begin{align*}
 \int_{\text{SIM}_{p-1}(t)} d\pmb{w_{p-1}}  \prod_{j=1}^{p-1} w_j^{-\frac{1+ \beta_j(1-2H_1)}{4H_0} }   &= \int_{\text{SIM}_{p-1}(t)} d\pmb{w_{p-1}}  \prod_{j=1}^{p-1} w_j^{\alpha_j} \notag \\
   &=  \frac{ t^{p-1+\alpha_1+\cdots +\alpha_{p-1}} }{\Gamma(p+\alpha_1+\cdots + \alpha_{p-1})} \prod_{i=1}^{p-1} \Gamma(1+\alpha_i); 
   \end{align*}
   see also Lemma 6.2 in \cite{SSX19}. As $ \beta_1+\cdots +\beta_{p-1}\in\{p-1,  p\}$, then 
   \[
    \sum_{j=1}^{p-1}\alpha_j \in\left\{  -(p-1) \frac{1-H_1}{2H_0},   \frac{1}{4H_0} - p \frac{1-H_1}{2H_0} \right\}   
   \]
   and 
   $$
   p+\alpha_1+...+ \alpha_{p-1} \in \left\{  p  \mathfrak{h}  + \frac{1-H_1}{2H_0}, p  \mathfrak{h}  + \frac{1}{4H_0} \right\} \,\,\,\text{with $ \mathfrak{h} = \frac{2H_0+H_1 -1}{2H_0} >\frac 14 $. }
   $$ 
   Thus,  
   \[
   \Gamma(p+\alpha_1+\cdots + \alpha_{p-1}) \geq  \Gamma(p\mathfrak{h} +  \frac{1-H_1}{2H_0}  ) \wedge \Gamma(p\mathfrak{h} +  \frac{1}{4H_0}  ) .
   \]
Then,  inequality \eqref{F6} follows from the previous computations.   The inequality \eqref{F8} has been proved in \cite[Remark 3.3]{SSX19}.
 \end{proof}
 
 \begin{remark}
Suppose that we choose   $\Theta(x)= G(a,x)$ in Lemma  \ref{LEM:H12}, for some $a\in (0,1)$. Then, if $\delta>0$ is such that $\varphi$ is bounded in $[-\delta, \delta]$, we obtain, under hypothesis  ${\bf (H1)}$ or  ${\bf (H2)}$,
 \[
 \int_{\R} dx G(a,x) \varphi(x)dx  \le \sup_{x\in [-\delta, \delta]} \varphi(x)+ \sqrt{a} \int_{\{ |x| \ge \delta\}} dx  \frac { \varphi(x)} {x^2},
 \]
 which  implies
 \begin{equation} \label{C0}
 \mathcal{C}_0= \sup_{a\in (0,1) }  \int_{\R} G(a,x) (1+\varphi(x))dx  <\infty.
 \end{equation}
\end{remark}

Next, we will give a useful estimate that will be applied in many places. For any $R>0$, we define
 \begin{align} \label{def:LR}
  \ell_R(\xi)=\frac{\sin^2{(R\xi)}}{\pi R\xi^2}
  \end{align}  and note that
  \[
  \int_{[-R,R]^2} dxdy e^{-\i (x-y)\xi} = \frac{4 \sin^2(R\xi)}{ \xi^2}=4\pi R\ell_R(\xi).
  \]
By \cite[Lemma 2.1]{NZ19BM}, $\{\ell_R\}_{ R>0}$ defines an approximation to the identity (as $R\to\infty$). 
\begin{lemma}\label{lemma-uni}
Assume {\bf (H1)} or {\bf (H2)}.  For any $\ep>0$, there exists a constant $C(\ep)$ such that 
 \begin{equation}\label{uni-bound}
 \int_{\R}\ell_R(\xi)\varphi(\xi)d\xi  \le  \ep +\frac {C(\ep)}R.
 \end{equation}
 \end{lemma}
 \begin{proof}
 In both cases {\bf (H1)} and {\bf (H2)}, the function $\varphi$ is continuous at $0$. Thus, for any $\e>0$, there exists some $\delta>0$ such that $\varphi(x) < \e$ for $|x|<\delta$. Then,
\begin{align*}
\int_{\{ | \xi | <\delta \}} \ell_R(\xi) \varphi(\xi) d\xi\leq \e\int_\R \ell_R(\xi) d\xi=\e,
\end{align*}
and
\[
\int_{\{ | \xi | \geq \delta \}}\ell_R(\xi) \varphi(\xi)  d\xi   \leq  
 \begin{dcases} {\displaystyle \frac{1}{\pi R} \int_{\{ | \xi | \geq \delta \}}  \frac{\varphi(\xi)}{\xi^2}    d\xi }\leq \frac{C(\e)}{R} ~\text{ in case ($\mathbf{H1}$)},\\
\quad\\
{\displaystyle \frac{1}{\pi R} \int_{\{ | \xi | \geq \delta \}} |\xi|^{-1-2H_1} d\xi }\le \frac{C(\e)}{R}  ~\text{ in case ($\mathbf{H2}$)}.
 \end{dcases}
\]
In this way, we just proved \eqref{uni-bound}.
\end{proof}

 \begin{remark}
If we choose   $\Theta(x)= \ell_R(x)$ in Lemma  \ref{LEM:H12},  then,  from Lemma \ref{lemma-uni}, we get  
 \begin{equation} \label{C1}
 \mathcal{C}_1= \sup_{R>1}  \int_{\R}  \ell_R(x)(1+\varphi(x))dx  <\infty.
 \end{equation}
\end{remark}

 \medskip
 
 Let $\Pi_p A_{t} (R)$ denote the projection of $A_t(R)$ on the $p$th Wiener chaos, where $A_t(R)$ is defined in \eqref{ATR}, that is, 
\[
\Pi_p A_{t} (R)=I^W_p\left(\int_{-R}^R f_{t,x,p}dx\right).
\]
  In view of Theorem \ref{thm00}, we need first to establish the asymptotic covariance of each chaos of $ \frac{1}{\sqrt{R}} A_t(R)$ and we know 
  that 
    \[
 \frac{1}{\sqrt{R}} A_t(R)=\sum_{p=1}^\infty \frac{1}{\sqrt{R}} \Pi_p A_{t} (R) =\sum_{p=1}^\infty  I^W_p\left( \frac{1}{\sqrt{R}}\int_{-R}^R f_{t,x,p}dx\right).
 \]

Now we are ready to present the proof of Proposition \ref{PROP:COV1}.

\begin{proof}[\rm\textbf{Proof of Proposition \ref{PROP:COV1} }]  In what follows, we only present the proof for the particular case where $t_1=t_2$, and the general case  follows from the same arguments with solely notational modifications. Let us fix $t>0$ and break the proof into three steps.

\medskip

 \noindent\textbf{Step 1: Proof of  \eqref{eqn-2}.} Noting that
$
\Pi_1 A_{t} (R)=I^W_1\left(\int_{-R}^R f_{t,x,1}dx\right)
$
and recalling  \cite[page 32]{NZ19BM}, we have
 \begin{align*}
{\rm Var}\Big[ \Pi_1  \big( A_t(R)\big) \Big]&  = \int_{-R}^R\int_{-R}^R   dxdy   \langle f_{s,x,1}, f_{t,y,1} \rangle_{\H}\\
&= \int_{[0,t]^2}drdv \gamma_0(r-v) \int_\R d\xi   \int_{-R}^R\int_{-R}^R  dxdy e^{-\i (x-y)\xi} \varphi(\xi) e^{-\frac{1}{2} (r+v) \xi^2}.
 \end{align*} 
We deduce from \eqref{def:LR} that
\[
\frac{1}{2 R}{\rm Var}\Big[ \Pi_1  \big( A_t(R)\big) \Big]= 2\pi\int_{[0,t]^2}drdv \gamma_0(r-v) \int_\R d\xi  \ell_R(\xi) \varphi(\xi) e^{-\frac{1}{2} (r+v) \xi^2}.
\]
 From Lemma \ref{uni-bound} and using that $\gamma_0$ is locally integrable, we get for any $\e>0$,
 \[
 \frac{1}{2 R}{\rm Var}\Big[ \Pi_1  \big( A_t(R)\big) \Big]  \le  C\left( \e + \frac {C(\e)} R\right),
 \]
which  implies $\lim\limits_{R\to\infty}R^{-1} {\rm Var}\Big[ \Pi_1  \big( A_t(R)\big) \Big] = 0$ in both cases {\bf (H1)} and {\bf (H2)}.

 \bigskip
 
\noindent\textbf{Step 2: Proof of  \eqref{eqn-1} and  \eqref{eqn-3}.}  Recall from \cite[page 31]{NZ19BM} that, for any $p\ge 2$,
 \begin{align*}
 \frac{1}{2R}{\rm Var}\Big[\Pi_p\big(A_t(R)\big)\Big] &=\frac{1}{2R}\int_{[-R,R]^2} dx dy\E\Big[I^W_p\left( f_{t,x,p}\right)I^W_p\left( f_{t,y,p}\right)\Big]\\
& =  p! \int_{\R} dz   \langle f_{t,z,p}, f_{t,0,p} \rangle_{\H^{\otimes p}} \frac{ \vert [-R,R] \cap [-z-R, -z+R] \vert  }{2R}
 \end{align*}
 and 
 \begin{align}
  \big\langle f_{t,z,p} ,  f_{t,0,p}  \big\rangle_{\H^{\otimes p}} 
&=\frac{1}{(p!)^2}  \int_{[0,t]^p\times  [0,t]^p} d\pmb{s_p} d\pmb{r_p}  \prod_{j=1}   \gamma_0(s_j-r_j ) \int_{\R^{p}} \mu(d\pmb{\xi_p}) e^{-\i z  \tau(\pmb{\xi_p})} \notag  \\&\qquad \times \E\left[  \prod_{j=1}^p \exp\left(  -\i   B_{s_j}  \xi_j     \right)  \right]  \E\left[  \prod_{j=1}^p \exp\left(  -\i  B_{r_j}  \xi_j     \right)  \right]. \label{acexp}
 \end{align}
It is not difficult to deduce from Proposition \ref{prop1} and the expression in \eqref{acexp} that
 \[
 (p!)^2 \big\langle f_{t,z,p}, f_{t,0,p} \big\rangle_{\H^{\otimes p}} =  \E\left[ \left(\mathcal{I}^{1,2}_{t, t}(z)\right)^{p} \right].
  \]
  Then 
  \[
\frac{1}{2R}   {\rm Var}\big[ \Pi_p A_t(R) \big] =\frac 1{p!} \int_{\R} dz    \E\left[ \left(\mathcal{I}^{1,2}_{t, t}(z)\right)^{p} \right]  \frac{ \vert [-R,R] \cap [-z-R, -z+R] \vert  }{2R}.
  \]
  Because  $ \vert [-R,R] \cap [-z-R, -z+R] \vert /(2R)$ converges to $1$ as $R$ tends to infinity, the convergences  \eqref{eqn-1} and  \eqref{eqn-3} will be a consequence of
  \begin{equation} \label{above}
\sum_{p=2} ^\infty  \frac 1{p!} \int_{\R} dz    \E\left[ \left|\mathcal{I}^{1,2}_{t, t}(z)\right|^{p} \right]  <\infty.
\end{equation}
  In view of Fatou's lemma and taking into account that  for any $z\in \R$,  $\mathcal{I}^{1,2}_{t, t}(z)$ is  the $L^k(\Omega)$-limit of  $\mathcal{I}^{1,2}_{t, t,\e}(z)$ (for any $k\ge 2$) as $\e$ tends to zero,
    to show \eqref{above}, it suffices to prove that
  \begin{equation}\label{above1}
 \sup_{\e>0}  \sum_{p= 2}^\infty \frac{1}{p!} \int_{\R} dz   \E\big[ \vert \mathcal{I}^{1,2}_{t, t,\e}(z)\vert ^p \big] <\infty.
   \end{equation}
   From
   \[
   \sum_{p=2}^\infty  \frac { |x|^p}{p!} = e^{|x|} - |x| -1 \le (e^x + e^{-x})-2 =  2 \sum_{n=1}^\infty  \frac { x^{2n} }{ (2n)!},
   \]
   we deduce  that \eqref{above1} holds true  provided
  \begin{equation}\label{above2}
 \sup_{\e>0}  \sum_{n= 1}^\infty \frac{1}{(2n)!} \int_{\R} dz   \E \left[\left( \mathcal{I}^{1,2}_{t, t,\e}(z)\right) ^{2n} \right] <\infty.
   \end{equation}
  \medskip
  
Next, let us prove \eqref{above2}.
First consider the case {\bf (H1)}. Fix an even integer $p = 2n\geq 2$.  
 For any $\e,a > 0$, by Fubini's theorem,
\begin{align}
T_{\e,a}(p):=&\, \int_{\R}\E\left[ \big(\mathcal{I}^{1,2}_{t, t,\e}(z)\big)^{p} \right]  \exp\left( -\frac{a}{2} z^2 \right) dz  \notag \\
=&\,  2\pi \int_{[0,t]^{2p} }  \int_{\R^{p}}  d\pmb{\xi_p} d\pmb{s_p} d\pmb{r_p}   \left( \prod_{j=1}^p \gamma_0(s_j - r_j ) \right) \left(  \prod_{j=1}^p  \varphi(\xi_j)  e^{-\e \xi_j^2} \right) \notag  \\ 
  &\qquad\qquad\quad  \times   G\big(a,  \tau(\pmb{\xi_p})  \big) \E\left[  \prod_{j=1}^p \exp\left(  -\i   B_{s_j}   \xi_j     \right)  \right]   \E\left[  \prod_{j=1}^p \exp\left(  -\i   B_{r_j}   \xi_j     \right)  \right].  \label{ec1}
\end{align}
 Note that $T_{\e,a}(p)\geq 0$ since $p$ is even. Using   
\[
\E\left[  \prod_{j=1}^p \exp\left(  -\i   B_{r_j}   \xi_j     \right)  \right]  \in (0, 1],
\]
we can  bound  $T_{\e,a}(p)$ as follows:
\begin{align*}
T_{\e,a}(p)&\leq 2\pi \Gamma_t^p  \int_{\R^{p}} d\pmb{\xi_p} \int_{[0,t]^p } d\pmb{s_p}  \left( \prod_{j=1}^p    \varphi(\xi_j)  \right)   G\big(a,  \tau(\pmb{\xi_p})  \big) \E\left[  \prod_{j=1}^p \exp\left(  -\i   B_{s_j}  \xi_j     \right)  \right] ,
\end{align*}
where the constant $\Gamma_t := \int_{-t}^t \gamma_0(r)dr$ is finite for each $t > 0$, since $\gamma_0$ is locally integrable. 
  By the inequality \eqref{F4} in  Lemma \ref{LEM:H12},
\[
\sup_{a\in(0,1)} T_{\e,a}(p )   \le \sup_{a\in(0,1)}  2\pi \Gamma_t^p K_{1,p} (G(a,\bullet),t) 
\le    \mathcal{C}_0 t\pi \Gamma_t^p p!(8\kappa_0 C_N)^{p-1} e^{\frac{tD_N}{2C_N}},
\]
where $\mathcal{C}_0$ is the constant defined in  \eqref{C0}.
Then, recalling $p=2n\geq 2$ and letting $a\downarrow 0$, we obtain
\begin{equation} \label{F15}
 \frac{1}{(2n)!}   \int_{\R}  \E\left[ \big(\mathcal{I}^{1,2}_{t, t,\e}(z)\big)^{2n} \right]dz \leq       \mathcal{C}_0 t\pi \Gamma_t^{2n} (8\kappa_0\Gamma_t C_N)^{2n-1} e^{\frac{tD_N}{2C_N}},
\end{equation}
which implies \eqref{above2}, provided we choose $N$ is such a way that $8\kappa_0\Gamma_t C_N<1$.

Now we consider  the case {\bf (H2)} where $\gamma_0(t) = | t|^{2H_0-2}$  and $\varphi(z) = | z|^{1-2H_1}$ with $H_0>1/2>H_1$ and  $H_0 + H_1>3/4$.      We begin with \eqref{ec1} and apply the embedding result   \eqref{F12} and Cauchy-Schwarz inequality  to write  for any $\e,a > 0$ and any even $p=2n\geq 2$,
\begin{align*}
T_{\e,a}(p) &\leq   2\pi  C_{H_0}^p  \left[    \int_{[0,t]^p} d\pmb{s_p} \left(  \int_{\R^{p}}  \mu(d\pmb{\xi_p})  G\big( a, \tau(\pmb{\xi_p})  \big)    \exp\left(   - \Var \sum_{j=1}^p   B_{s_j}   \xi_j     \right)   \right)^{\frac{1}{2H_0}}    \right]^{2H_0}\\
   &=  2\pi C_{H_0}^p  [K_{2,p}(G(a, \bullet),t)] ^{2H_0},
   \end{align*}
   where  $K_{2,p}(G(a, \bullet),t)$ has been defined in \eqref{K2}. 
Then, applying inequality \eqref{F6} of  Lemma \ref{LEM:H12} leads to
\[
T_{\e,a}(p)\leq 2\pi  C_{H_0}^p  (p!  \mathcal{C}_0 )^{2H_0}  \frac{C_1^{2H_0p}  \left(t^{2H_0p\mathfrak{h} +1-H_1 }\vee t^{2H_0p\mathfrak{h} +  \frac{1}{2}}  \right)}{\Gamma( p\mathfrak{h} +  \frac{1-H_1}{2H_0}   )^{2H_0} \wedge \Gamma(p\mathfrak{h} + \frac{1}{4H_0} )^{2H_0}    },
\]
 for $\e>0$, $a\in(0,1)$ and any even integer $p=2n\geq 2$, where $\mathfrak{h} = \frac{2H_0 +H_1 -1}{2H_0}>\frac 14$ and
 with $\mathcal{C}_0$ being the constant defined in  \eqref{C0}.
  Thus,  letting $a\downarrow 0$, we have for $n\geq 1$,
  \begin{align}
\frac 1{ (2n)!}   \int_{\R} \E\Big[\big( \mathcal{I}^{1,2}_{t, t,\e}(z)\big)^{2n}\Big] dz \leq  \frac{  C_2^{2n}((2n)!)^{2H_0-1}\left(t^{4H_0n\mathfrak{h} +1-H_1 }\vee t^{4H_0n\mathfrak{h} +  \frac{1}{2}}  \right)}{\Gamma( 2n\mathfrak{h} +  \frac{1-H_1}{2H_0}   )^{2H_0} \wedge \Gamma(2n\mathfrak{h} + \frac{1}{4H_0} )^{2H_0}      }, \label{view331}
  \end{align}
  for some constant $C_2$, which depends only on $H_0$ and $H_1$.
  Notice that  the denominator in  the right-hand side of \eqref{view331} behaves as $[(2n)!]^{ 2H_0 \mathfrak{h}}$, as $n\to \infty$,  and 
 $
  2H_0 \mathfrak{h} =  2H_0+H_1-1.
  $
   Thus, the   estimate \eqref{view331}   implies \eqref{above2} in the case {\bf (H2)}. 
 Hence the proof of Proposition  \ref{PROP:COV1} is completed now.
 \end{proof}
 
 \begin{remark}
We observe
 \[
\frac{1}{2R}  {\rm Var}\big( A_{t} (R) \big)=\frac{1}{2R}{\rm Var} \left[\int_{-R}^R (u(t,x)-1 )  dx\right] = \frac{1}{2R} \int_{[-R,R]^2} \big(  \E\big[ u(t,x) u(t,y) ] - 1 \big) dxdy,
  \]
 which in view of Proposition \ref{prop1}, is equal to
\[
\frac{1}{2R} \int_{[-R,R]^2} \E\big[( e^{\mathcal{I}^{1,2}_{t, t}(x-y)} -1 ) \big]dxdy = \int_{\R}  \frac{\vert [z-R,z+R]\cap [-R,R] \vert}{2R} \E\big[  e^{\mathcal{I}^{1,2}_{t, t}(z)} -1\big]dz.
\]
As $R\to \infty$, this converges to
\[
 \int_{\R}\E\big[ e^{ \mathcal{I}^{1,2}_{t, t}(z)} -1 \big] dz= \int_{\R} \sum_{p= 2}^\infty \frac{1}{p!} \E\left[\left( \mathcal{I}^{1,2}_{t, t}(z)\right)^p\right] dz,
 \]
 since $\E\big[\mathcal{I}^{1,2}_{t, t}(z)\big]=\lim\limits_{\e\downarrow 0}\E\big[\mathcal{I}^{1,2}_{t, t,\e}(z)\big]=0$, based on  {\bf Step 1} in the proof of Proposition \ref{PROP:COV1}. The above limit is exactly the one in \eqref{eqn-1}.
  \end{remark}

\section{Convergence of finite-dimensional distributions}\label{SEC4}
 
 To apply the multivariate chaotic CLT (Theorem \ref{thm00}), we need to show the convergence of finite-dimensional distributions. There are four conditions in this theorem that we need to check. In fact, we only need to verify the condition (c), as the other conditions are satisfied as a consequence of our estimates in  the previous section. We refer readers to \cite{NZ19BM} for similar arguments. 
 
We can write 
  \[
 \frac{1}{\sqrt{R}} A_t(R)=\sum\limits_{p= 1}^\infty  I_p^W(g_{p,R}(t)),
 \]
where 
 \[
 g_{p,R}(t) := \frac{1}{\sqrt{R}} \int_{-R}^R f_{t,x,p}dx.
 \]
 \begin{proposition} \label{prop2}
For each integer $p\geq 2$ and each integer $1\leq r \leq p-1$, we have
 \[
 \big\| g_{p,R}(t) \otimes_r g_{p,R}(t) \big\| _{\H^{\otimes 2p-2r}} \xrightarrow{R\to\infty} 0.
 \]
 \end{proposition}

 \begin{proof}
 Along the proof $C$ will be a generic constant that may very form line to line.
 We put 
 \[
 \mathfrak{f}(\pmb{s_p}, \pmb{y_p}) = f_{t,0,p}(\pmb{s_p}, \pmb{y_p})
 \]
  and recall that,   with $B$ a real-valued standard Brownian motion on $\R$,
 \begin{equation}
 (\F \mathfrak{f})(\pmb{s_p}, \pmb{\xi_p}) = (\F f_{t,0,p})(\pmb{s_p}, \pmb{\xi_p})  \label{expform}
  = \frac{1}{p!}  \E\left[ \exp\Big( -\i \sum_{j=1}^p B_{s_j }   \xi_j\Big) \right], 
 \end{equation}
   where $\F \mathfrak{f}$ stands for the Fourier transform with respect to the spatial variables. As a consequence,  $ (\F \mathfrak{f})(\pmb{s_p}, \pmb{\xi_p})$ is a positive, bounded and uniformly continuous function in $\pmb{\xi_p}$.

 Now we write, 
 with the notation $d\ell^{4p} = d\pmb{s_r}  d\pmb{\wt{s}_r}d\pmb{v_r} d\pmb{\wt{v}_r} d\pmb{t_{p-r}} d\pmb{\wt{t}_{p-r}}d\pmb{w_{p-r}} d\pmb{\wt{w}_{p-r}}  $ for the Lebesgue measure on $[0,t]^{4p}$, 
 \begin{align*}
 \big\| g_{p,R}(t) \otimes_r g_{p,R}(t) \big\| _{\mathfrak{H}^{\otimes (2p-2r)}}^2  &= (2\pi)^2\int_{[0,t]^{4p}} d\ell^{4p}  \left( \prod_{i=1}^r \gamma_0(s_i-\wt{s}_i)\gamma_0(v_i-\wt{v}_i)\right) \\
& \qquad  \times \left( \prod_{j=1}^{p-r} \gamma_0(t_j-\wt{t}_j)\gamma_0(w_j-\wt{w}_j)\right)  \mathcal{J}_R,
 \end{align*}
 with $\mathcal{J}_R = \mathcal{J}_R\big( \pmb{s_r},  \pmb{\wt{s}_r}, \pmb{v_r} , \pmb{\wt{v}_r}, \pmb{t_{p-r}}, \pmb{\wt{t}_{p-r}},\pmb{w_{p-r}} ,\pmb{\wt{w}_{p-r}}  \big)$ given by
 \begin{align*}
 \mathcal{J}_R &=  \int_{\R^{2p}}\mu(d\pmb{\xi_r}) \, \mu(d\pmb{\wt{\xi}_r})  \mu(d\pmb{\eta_{p-r}})  \mu(d\pmb{\wt{\eta}_{p-r}})  \\
&\quad \times (\F \mathfrak{f})(\pmb{s_r}, \pmb{t_{p-r}},\pmb{\eta_{p-r}},\pmb{\xi_r})   (\F \mathfrak{f})(\pmb{\wt{s}_r}, \pmb{w_{p-r}}, \pmb{\wt{\eta}_{p-r}}, - \pmb{\xi_r})| a + b |^{-1/2}  |a-\wt{b}|^{-1/2}    \\
&\quad     \times (\F \mathfrak{f}) (\pmb{v_r}, \pmb{\wt{t}_{p-r}}, \pmb{-\eta_{p-r}},\pmb{\wt{\xi}_r})   (    \F \mathfrak{f})(\pmb{\wt{v}_r},\pmb{\wt{w}_{p-r}},    - \pmb{\wt{\eta}_{p-r}} , -\pmb{\wt{\xi}_r})  | \wt{a} - b |^{-1/2} |\wt{ a} + \wt{b} |^{-1/2}  \\
&\quad  \times J_{1/2} \big( R | a + b | \big) J_{1/2}\big( R|a-\wt{b}| \big)J_{1/2}\big( R| \wt{a} - b | \big) J_{1/2}\big( R |\wt{ a} + \wt{b} | \big), 
 \end{align*}
where we use the following short-hand notation
  \begin{center}
  $a = \tau(\pmb{\xi_r}),  b=\tau(\pmb{\eta_{p-r}}), \wt{a} = \tau(\pmb{\wt{\xi}_r})$, $ \wt{b}=\tau(\pmb{\wt{\eta}_{p-r}})$,
 \end{center}
 and 
 \[
 J_{1/2}(x)=\sqrt{\frac{2}{\pi}}\ \frac{\sin{(x)}}{\sqrt{x}}, \  x\in \mathbb{R}_+
 \]
 is the Bessel function of first kind with order $1/2$.
 This is obtained in the same way as in \cite{NZ19BM} and we refer readers to this reference for more details. 

  Now we decompose  the integral in the spatial variable into two parts, and we write  for any given $\delta>0$,
 \[
 \mathcal{J}_R = \mathcal{J}_{1,R,\delta} +  \mathcal{J}_{2,R,\delta} := \int_{\R^{2p}} {\bf 1}_{\{ | a+ b| \geq \delta \}} + \int_{\R^{2p}} {\bf 1}_{\{ | a+ b| < \delta \}}.
 \]
 This leads to the decomposition
 \[
  \big\| g_{p,R}(t) \otimes_r g_{p,R}(t) \big\| _{\mathfrak{H}^{\otimes (2p-2r)}}^2  = (2\pi)^2( \mathcal{Y}_{1,R,\delta} +  \mathcal{Y}_{2,R,\delta}),
  \]
  where, for $k=1,2$,
  \[
 \mathcal{Y}_{k,R,\delta}:=  \int_{[0,t]^{4p}}  d\ell^{4p}  \left( \prod_{i=1}^r \gamma_0(s_i-\wt{s}_i)\gamma_0(v_i-\wt{v}_i)\right) \times \left( \prod_{j=1}^{p-r} \gamma_0(t_j-\wt{t}_j)\gamma_0(w_j-\wt{w}_j)\right)  \mathcal{J}_{k,R,\delta}.
\]
 The proof will be done in two steps.
 
 \medskip
 \noindent
 {\bf Step 1}:  We will show that for any fixed  $\delta >0$,   $\mathcal{Y}_{1,R,\delta}$ tends to zero as  $R\to \infty$.
First we   apply Cauchy-Schwarz inequality several times to get
 \begin{align*}
 \mathcal{J}_{1,R,\delta} &\leq 4 \left( \int_{ \{ | \tau(\pmb{\xi_p}) | \geq \delta   \} } \ell_R( \tau(\pmb{\xi_p})  )  \big\vert\F \mathfrak{f} \big\vert^2(\pmb{s_r}, \pmb{t_{p-r}},\pmb{\xi_p})   \mu(d\pmb{\xi_p})  \right)^{1/2} \\
 &\quad \times \left( \int_{  \R^{p}} \ell_R( \tau(\pmb{\xi_p})  )  \big\vert\F \mathfrak{f} \big\vert^2(\pmb{\wt{v}_r}, \pmb{\wt{w}_{p-r}},\pmb{\xi_p})   \mu(d\pmb{\xi_p})  \right)^{1/2} \\
  &\quad \times \left( \int_{  \R^{p}} \ell_R( \tau(\pmb{\xi_p})  )  \big\vert\F \mathfrak{f} \big\vert^2(\pmb{\wt{s}_r}, \pmb{w_{p-r}},\pmb{\xi_p})   \mu(d\pmb{\xi_p})  \right)^{1/2} \\
   &\quad \times \left( \int_{  \R^{p}} \ell_R( \tau(\pmb{\xi_p})  )  \big\vert\F \mathfrak{f} \big\vert^2(\pmb{v_r}, \pmb{\wt{t}_{p-r}},\pmb{\xi_p})   \mu(d\pmb{\xi_p})  \right)^{1/2},
 \end{align*}
 where $ \ell_R(x)=\frac{1}{2} |x|^{-1} J_{1/2}^2(R|x|)$ is introduced in \eqref{def:LR}.
 We will prove separately for cases $\mathbf{(H1)}$ and $\mathbf{(H2)}$ that  $\mathcal{Y}_{1,R,\delta} \to 0$ as $R\to \infty$.  \\

 \noindent{\textbf{Proof of $\mathcal{Y}_{1,R,\delta}\xrightarrow{R\to\infty} 0$ under} $\mathbf{(H1)}$:} 
By Cauchy-Schwarz  inequality again applied to the integration in time, we get
  \begin{align}
\mathcal{Y}_{1,R,\delta}\leq&  4\Bigg\{ \int_{[0,t]^{4p}} d\ell^{4p}  \left( \prod_{i=1}^r \gamma_0(s_i-\wt{s}_i)\gamma_0(v_i-\wt{v}_i)\right) \notag  \left( \prod_{j=1}^{p-r} \gamma_0(t_j-\wt{t}_j)\gamma_0(w_j-\wt{w}_j)\right) \\
& \qquad \times  \left( \int_{  \R^{p}} \ell_R( \tau(\pmb{\xi_p})  )  \big\vert\F \mathfrak{f} \big\vert^2(\pmb{\wt{v}_r}, \pmb{\wt{w}_{p-r}},\pmb{\xi_p})   \mu(d\pmb{\xi_p})  \right)     \notag \\
&\qquad \qquad\qquad \times \int_{ \{ |\tau(\pmb{\xi_p})| \geq \delta   \} } \ell_R( \tau(\pmb{\xi_p})  )  \big\vert\F \mathfrak{f} \big\vert^2(\pmb{s_r}, \pmb{t_{p-r}},\pmb{\xi_p})   \mu(d\pmb{\xi_p})  \Bigg\}^{1/2}    \notag\\
&\times  \Bigg\{ \int_{[0,t]^{4p}} d\ell^{4p}    \left( \prod_{i=1}^r \gamma_0(s_i-\wt{s}_i)\gamma_0(v_i-\wt{v}_i)\right) \notag \left( \prod_{j=1}^{p-r} \gamma_0(t_j-\wt{t}_j)\gamma_0(w_j-\wt{w}_j)\right) \\
& \qquad \times\left( \int_{  \R^{p}} \ell_R( \tau(\pmb{\xi_p})  )  \big\vert\F \mathfrak{f} \big\vert^2(\pmb{\wt{s}_r}, \pmb{w_{p-r}},\pmb{\xi_p})   \mu(d\pmb{\xi_p})  \right)  \notag \\
&\qquad \qquad\qquad \times  \int_{ \R^{p}} \ell_R( \tau(\pmb{\xi_p})  )  \big\vert\F \mathfrak{f} \big\vert^2(\pmb{v_r}, \pmb{\wt{t}_{p-r}},\pmb{\xi_p})   \mu(d\pmb{\xi_p})  \Bigg\}^{1/2}  \notag \\
& =:  4 V_{1,R,\delta} ^{1/2}     V_{2,R}  ^{1/2}.  \notag
 \end{align}
 For the term $V_{1,R,\delta}$,
 we have the estimate
  \begin{align}  \notag
 V_{1,R,\delta} &\leq   \Gamma_t^{2p} \left[   \int_{[0,t]^p} d\pmb{s_p}   \int_{ \{ | \tau(\pmb{\xi_p}) | \geq \delta   \} } \ell_R( \tau(\pmb{\xi_p})  )  \big\vert\F \mathfrak{f} \big\vert^2(\pmb{s_p},\pmb{\xi_p})   \mu(d\pmb{\xi_p})  \right]\\   \label{e1}
 &\qquad  \times    \left(\int_{[0,t]^{p}} d\pmb{t_p} \int_{  \R^{p}} \ell_R( \tau(\pmb{\xi_p})  )  \big\vert\F \mathfrak{f} \big\vert^2(\pmb{t_p},\pmb{\xi_p})   \mu(d\pmb{\xi_p}) \right)  =:   \Gamma_t^{2p} V_{11,R,\delta}  V_{12,R},
 \end{align}
 where we recall that $\Gamma_t= \int_{-t} ^t \gamma_0(s) ds$.
   We will prove that $V_{12,R}$ is uniformly bounded and $ V_{11,R,\delta}$ vanishes asymptotically as $R\rightarrow \infty$. 
  In view of   \eqref{expform},  making the   change of variables $t_j = t - s_j$ and $\eta_j =  \xi_1 +\dots + \xi_j $ for each $j=1,\dots, p$, with $\eta_0=0$, we obtain, in view of \eqref{F2},
  \begin{align*}
 V_{12,R}& =\frac{1}{p!} \int_{\Delta_p(t)}d\pmb{s_p}  \int_{  \R^{p}} \mu(d\pmb{\xi_p})\ell_R( \tau(\pmb{\xi_p})  ) \exp\left( - \sum_{j=1}^p (s_j -s_{j+1}) ( \xi_1 + \dots+ \xi_j )^2  \right) \\
&= \frac{1}{p!} \int_{\R} d\eta_p \ell_R(\eta_p)  \int_{\R^{p-1}}d\pmb{\eta_{p-1}} \int_{\text{SIM}_p(t)}d\pmb{w_p}    \prod_{j=1}^p e^{-w_j \eta_j^2  } \varphi(\eta_j - \eta_{j-1}) \\
&  =K_{1,p} (\ell_R,t).
\end{align*}
By inequality \eqref{F4} in Lemma \ref{LEM:H12} and \eqref{C1},   this implies
\begin{equation} \label{e2}
\sup_{R>1} V_{12,R} < \infty.
\end{equation}
     In the same way, we have,  for any $\e>0$, using \eqref{F4} and \eqref{uni-bound},
\begin{align}  \notag
 V_{11,R,\delta}  &\le   K_{1,p} (\ell_R \mathbf{1}_{(-\delta, \delta)^c},t) \le  C
   \int_{\{ | \eta|\geq \delta \}  } d \eta \ell_R( \eta)  (1+ \varphi(\eta)) \\  \label{F13}
& \le  C  \int_{\{ |\eta| \ge R \delta\}} d\eta \frac {\sin^2(\eta)}{ \pi \eta^2} + C\left( \e + \frac {C(\e)} R \right).
 \end{align}
 for some constant $C>0$.  So, $\lim_{R\rightarrow \infty} V_{11,R,\delta} =0$ and,  therefore, from \eqref{F13},  \eqref{e1} and \eqref{e2} we have proved  that
\[
V_{1,R,\delta} \xrightarrow{R\to\infty} 0.
\]
For the term $V_{2,R}$, note that
 $
 V_{2,R} \leq  \Gamma_t^{2p} V_{12,R}^2,
 $
so  we have the uniform boundedness of $V_{2,R}$ over   $R>1$. Thus,  we completed the proof of $\mathcal{Y}_{1,R,\delta}\xrightarrow{R\to\infty} 0$ under $\mathbf{(H1)}$.

\medskip

 \noindent{\textbf{Proof of $\mathcal{Y}_{1,R,\delta}\xrightarrow{R\to\infty} 0$ under} $\mathbf{(H2)}$:}   Using the embedding result \eqref{F12} and Cauchy-Schwarz inequality, we can write
 \begin{align*}
 \mathcal{Y}_{1,R,\delta}& \leq   C_{H_0}^p\Bigg[ \int_{[0,t]^{2p}}d\pmb{s_p}d\pmb{t_p}\left( \int_{   \{ |\tau(\pmb{\xi_p})| \geq \delta   \}} \ell_R( \tau(\pmb{\xi_p})  )  \big\vert\F \mathfrak{f} \big\vert^2(\pmb{t_p},\pmb{\xi_p})   \mu(d\pmb{\xi_p}) \right)^{\frac{1}{2H_0}} \\
 &\qquad\qquad\times \left( \int_{  \R^{p}} \ell_R( \tau(\pmb{\xi_p})  )  \big\vert\F \mathfrak{f} \big\vert^2(\pmb{s_p},\pmb{\xi_p})   \mu(d\pmb{\xi_p})   \right)^{\frac{1}{2H_0}}\Bigg]^{H_0}\\
&\quad\times \Bigg[\int_{[0,t]^{2p}}d\pmb{s_p}d\pmb{t_p}\left( \int_{ \R^p} \ell_R( \tau(\pmb{\xi_p})  )  \big\vert\F \mathfrak{f} \big\vert^2(\pmb{t_p},\pmb{\xi_p})   \mu(d\pmb{\xi_p}) \right)^{\frac{1}{2H_0}} \\
 &\qquad\qquad\times \left( \int_{  \R^{p}} \ell_R( \tau(\pmb{\xi_p})  )  \big\vert\F \mathfrak{f} \big\vert^2(\pmb{s_p},\pmb{\xi_p})   \mu(d\pmb{\xi_p})   \right)^{\frac{1}{2H_0}}\Bigg]^{H_0}\\
 &=   C_{H_0}^pT_{1,R,\delta}^{H_0} T_{2,R}^{3H_0},
 \end{align*}
where
\begin{align*}
T_{1,R,\delta} :=&  \int_{[0,t]^{p}}d\pmb{t_p}\left( \int_{   \{ |\tau(\pmb{\xi_p})| \geq \delta   \}} \ell_R( \tau(\pmb{\xi_p})  )  \big\vert\F \mathfrak{f} \big\vert^2(\pmb{t_p},\pmb{\xi_p}),  \mu(d\pmb{\xi_p}) \right)^{\frac{1}{2H_0}}   ~\text{and} \\
 T_{2,R}: =& \int_{[0,t]^{p}}d\pmb{t_p}\left( \int_{  \R^p} \ell_R( \tau(\pmb{\xi_p})  )  \big\vert\F \mathfrak{f} \big\vert^2(\pmb{t_p},\pmb{\xi_p})   \mu(d\pmb{\xi_p}) \right)^{\frac{1}{2H_0}}. 
 \end{align*}
From \eqref{expform}, a change of variables and   inequalities \eqref{F6} and  \eqref{C1}, we can easily get
 \begin{align}  \notag 
T_{2,R}= & (p!)^{-\frac{1}{H_0}}  \int_{[0,t]^p}d\pmb{s_p}\left(\int_{\R^p}\mu(d\pmb{\xi_p})\ell_R(\tau(\pmb{\xi_p}))\exp\left[-\text{Var}\sum_{j=1}^pB_{s_j}\xi_j\right]\right)^{\frac{1}{2H_0}}\\  \label{eqn-V2}
&= (p!)^{-\frac{1}{H_0}}  K_{2,p} (\ell_R, t) \le C
\end{align}
for all $R>1$. The term  $ T_{1,R,\delta}$ can be estimated as follows
\begin{align*}
 T_{1,R,\delta} &= (p!)^{-\frac{1}{H_0}}  \int_{[0,t]^p}d\pmb{s_p}\left(\int_{ \{ |\eta_p| \ge \delta\} }\mu(d\pmb{\xi_p})\ell_R(\tau(\pmb{\xi_p}))\exp\left[-\text{Var}\sum_{j=1}^pB_{s_j}\xi_j\right]\right)^{\frac{1}{2H_0}} \\
 &=  (p!)^{-\frac{1}{H_0}}  K_{2,p} (\ell_R \mathbf{1}_{(-\delta, \delta)^c}, t)\\
&\le 
  C \int_{\{|\eta|\geq \delta\}}d\eta \ell_R(\eta) (1+       |\eta|^{1-2H_1}  ),
\end{align*}
which converges to zero as $R\to \infty$ as we have already noted.
  Therefore, combining the calculations on $T_{1,R,\delta}$ and $T_{2,R}$, we  show that $\mathcal{Y}_{1,R,\delta}\xrightarrow{R\to\infty} 0$ under $\mathbf{(H2)}$.

 \medskip
 \noindent
 {\bf Step 2}:  We will show that  
 \begin{equation} \label{ECU8}
 \lim_ {\delta \rightarrow 0} \limsup_{R\rightarrow  \infty} \mathcal{Y}_{2,R,\delta} =0.
 \end{equation}
 Using Cauchy-Schwarz multiple times and changing of variables, we  obtain
 \begin{align}
 &\mathcal{J}_{2,R,\delta} \leq  4 \int_{\{  | a+b| < \delta  \}}  \mu(d\pmb{\xi_r})  \mu(d\pmb{\eta_{p-r}}) \sqrt{ \ell_R(a+b) } ~ \F\mathfrak{f} \big(  \pmb{s_r}, \pmb{t_{p-r}},   \pmb{ \eta_{p-r}}, \pmb{\xi_r} \big)\notag \\
& \quad  \times \left(   \int_{\R^{p}} \mu(d\pmb{\wt{\xi}_r})  \mu(d\pmb{\wt{\eta}_{p-r}})  \ell_R(\wt{a}+\wt{b}) \big\vert \F\mathfrak{f} \big\vert^2\big(  \pmb{\wt{v}_r}, \pmb{\wt{w}_{p-r}},   \pmb{ \wt{\eta}_{p-r}}, \pmb{\wt{\xi}_r} \big)   \right)^{1/2}  \notag   \\
& \quad \times  \left[   \int_{\R^{p}} \mu(d\pmb{\wt{\xi}_r})\mu(d\pmb{\wt{\eta}_{p-r}}) \ell_R(\wt{a}+ b) \ell_R(a+\wt{b}) \big\vert\F \mathfrak{f} \big\vert^2\big(  \pmb{v_r}, \pmb{\wt{t}_{p-r}},   \pmb{ \eta_{p-r}}, \pmb{\wt{\xi}_r} \big)  \big\vert\F \mathfrak{f}\big\vert^2\big(  \pmb{\wt{s}_r}, \pmb{w_{p-r}}, \pmb{\wt{\eta}_{p-r}}, \pmb{\xi_r}\big)   \right]^{1/2}\notag \\
&   \qquad   \leq 4\Bigg[  \left(   \int_{\R^{p}} \mu(d\pmb{\wt{\xi}_p})  \ell_R(\tau(\pmb{\wt{\xi}_p})) \big\vert \F\mathfrak{f} \big\vert^2\big(  \pmb{\wt{v}_r}, \pmb{\wt{w}_{p-r}},   \pmb{\wt{\xi}_p} \big)   \right)  \notag \\
&  \quad \times  \left( \int_{\{  | \tau(\pmb{\xi_p}) | < \delta  \}}  \mu(d\pmb{\xi_p})   \ell_R( \tau(\pmb{\xi_p}))  \big\vert \F\mathfrak{f}\big\vert^2 \big(  \pmb{s_r}, \pmb{t_{p-r}},   \pmb{\xi_p} \big) \right)\Bigg]^{1/2} \notag\\
& \quad \times \Bigg[  \int_{ \{ | a+b| < \delta\} \times\R^{p}}  \mu(d\pmb{\xi_r})  \mu(d\pmb{\eta_{p-r}})   \mu(d\pmb{\wt{\xi}_r})  \mu(d\pmb{\wt{\eta}_{p-r}}) \notag  \\
&\quad \times \big\vert\F \mathfrak{f}\big\vert^2 \big(  \pmb{v_r}, \pmb{\wt{t}_{p-r}},   \pmb{\eta_{p-r}}, \pmb{\wt{\xi}_r} \big)   \big\vert \F\mathfrak{f}\big\vert^2 \big(  \pmb{\wt{s}_r}, \pmb{w_{p-r}},   \pmb{\wt{ \eta}_{p-r}}, \pmb{\xi_r} \big)       \ell_R(\wt{a}+ b) \ell_R(a+\wt{b})        \Bigg]^{1/2}\notag\\
& \qquad =: 4U_{1,R,\delta} ^{1/2} U_{2,R,\delta} ^{1/2}.    \label{F13-1}
  \end{align}
  
   In what follows, we are going to  prove separately in case $(\mathbf{H1})$ and in case $(\mathbf{H2})$ that \eqref{ECU8} holds.

\medskip
 \noindent{\bf Proof of  \eqref{ECU8} under  $\mathbf{(H1)}$}:  From   \eqref{F13-1}, using Cauchy-Schwarz inequality again for the integration in time, we have
 \[ 
     \mathcal{Y}_{2,R,\delta} \leq 4  \sqrt{  \mathfrak{X}_{1,R,\delta}  \mathfrak{X}_{2,R,\delta}  }\, ,
 \]
 where, for $k=1,2$,
\[
  \mathfrak{X}_{k,R,\delta} := \int_{[0,t]^{4p}}  d\ell^{4p}  \left( \prod_{i=1}^r \gamma_0(s_i-\wt{s}_i)\gamma_0(v_i-\wt{v}_i)\right) \left( \prod_{j=1}^{p-r} \gamma_0(t_j-\wt{t}_j)\gamma_0(w_j-\wt{w}_j)\right)  U_{k,R,\delta}.
  \]
 One can show by the same arguments as before that, for all $R>1$ and $\delta>0$,
 \begin{align*}
 \mathfrak{X}_{1,R,\delta}\leq \Gamma_t^{2p}V_{12,R}^2\leq C.
 \end{align*}
  Now we write 
  \begin{align*}
  \mathfrak{X}_{2,R,\delta} &\leq \Gamma_t^{2p}\int_{[0,t]^{2p}} d\pmb{\wt{s_r} }d\pmb{\wt{t}_{p-r} } d\pmb{v_r}d\pmb{w_{p-r}}  \int_{ \{ | a+b| < \delta\} \times\R^{p}}  \mu(d\pmb{\xi_r})  \mu(d\pmb{\eta_{p-r}})   \mu(d\pmb{\wt{\xi}_r}) \mu(d\pmb{\wt{\eta}_{p-r}})     \\
&\qquad \times   \big\vert\F \mathfrak{f}\big\vert^2 \big(  \pmb{v_r}, \pmb{\wt{t}_{p-r}},   \pmb{\eta_{p-r}}, \pmb{\wt{\xi}_r} \big)   \big\vert \F\mathfrak{f}\big\vert^2 \big(  \pmb{\wt{s}_r}, \pmb{w_{p-r}},   \pmb{\wt{ \eta}_{p-r}}, \pmb{\xi_r} \big)     \ell_R(\wt{a}+ b) \ell_R(a+\wt{b}) \\
&=\Gamma_t^{2p}   \int_{ \R^{2p}}  \mu(d\pmb{\xi_p})  \mu(d\pmb{\wt{\xi}_{p}})  {\bf 1}_{ \{  | \xi_1 + \dots+ \xi_r  + \wt{\xi}_{r+1} + \dots+ \wt{\xi}_{p}   | < \delta\} }  \ell_R\big( \tau(\pmb{\xi_p}) \big)    \ell_R\big( \tau(\pmb{\wt{\xi}_p}) \big)   \\  
&\qquad  \times  \left( \int_{[0,t]^p} d\pmb{s_p}    \big\vert\F \mathfrak{f}\big\vert^2 \big(  \pmb{s_p},    \pmb{\wt{\xi}_p} \big) \right) \left(  \int_{[0,t]^p}  d\pmb{t_{p}}     \big\vert\F \mathfrak{f}\big\vert^2 \big(  \pmb{t_p},   \pmb{\xi_p} \big) \right)    . 
          \end{align*}
Using  \eqref{expform} and a change of variables in time,  we can rewrite the last  expression  as follows, with $\eta_0=\wt{\eta}_0=0$, 
\begin{align} 
\mathfrak{X}_{2,R,\delta} & \le  \frac{\Gamma_t^{2p}}{(p!)^4}   \int_{ \R^{2p}}  \mu(d\pmb{\xi_p})  \mu(d\pmb{\wt{\xi}_{p}})  {\bf 1}_{ \{  | \xi_1 + \dots+ \xi_r  + \wt{\xi}_{r+1} + \dots+ \wt{\xi}_{p}   | < \delta\} } \ell_R\big( \tau(\pmb{\xi_p}) \big)    \ell_R\big( \tau(\pmb{\wt{\xi}_p}) \big) \notag  \\  
&  \qquad\qquad \times    \int_{[0,t]^{2p}} d\pmb{s_p}   d\pmb{t_p}   \exp\Big( -{\rm Var} \sum_{j=1}^p   B_{s_j }   \wt{\xi}_j\Big)   \exp\Big( -{\rm Var} \sum_{j=1}^p   B_{t_j }   \xi_j\Big) \notag\\
&=  \frac{\Gamma_t^{2p}}{(p!)^2}  \int_{ \R^{2p}}d\pmb{\eta_p}d\pmb{\wt{\eta}_p}\left(  \prod_{j=1}^p\varphi(\eta_j-\eta_{j-1})\varphi(\wt{\eta}_j-\wt{\eta}_{j-1})  \right)   {\bf 1}_{ \{  | \eta_r + \wt{\eta}_p-\wt{\eta}_r  | < \delta\} } \ell_R\big( \eta_p\big)    \ell_R\big( \wt{\eta}_p \big) \notag  \\  
&\qquad \qquad \times    \int_{\text{SIM}_p(t)^{2}} d\pmb{w_p}   d\pmb{\wt{w}_p} \prod_{j=1}^pe^{-w_j\eta_j^2}e^{-\wt{w}_j\wt{\eta}^2_j}\notag\\
&\leq   \int_{G(\delta)}  d\pmb{\eta_p}   d\pmb{\wt{\eta}_p}   \Phi(\pmb{\eta_p}, \pmb{\wt{\eta}_p}), \notag
\end{align}
where  $G(\delta)=\{ (\pmb{\eta_p}, \pmb{\wt{\eta}_p}) \in \R^{2p}: | \eta_r + \wt{\eta}_p-\wt{\eta}_r  | < \delta\}$ and
\begin{align*}
\Phi(\pmb{\eta_p}, \pmb{\wt{\eta}_p}) &:
= \frac{\Gamma_t^{2p}\kappa_0^{2(p-1)}}{(p!)^2}  \sum_{\pmb{\beta},\, \pmb{\wt{\beta}}\in\mathcal{A}_p}  \left(  \prod_{j=1}^p\varphi(\eta_j)^{\beta_j}\varphi(\wt{\eta}_j)^{\wt{\beta}_j}  \right)   \ell_R\big( \eta_p\big)    \ell_R\big( \wt{\eta}_p \big) \notag  \\
& \qquad \times     \int_{\text{SIM}_p(t)^{2}} d\pmb{w_p}   d\pmb{\wt{w}_p} \prod_{j=1}^pe^{-w_j\eta_j^2}e^{-\wt{w}_j\wt{\eta}^2_j}
 \end{align*}
 Then, we decompose the set $G(\delta)$ as follows
 \[
 G(\delta) \subset \{  | \eta_r  -\wt{\eta}_r  |   <2\delta\} \cup \{ |\wt{\eta}_p | \ge \delta\}.
 \]
  From the estimation \eqref{F13}  for the term $V_{11,R,\delta}$, it follows that, for any $\delta>0$, 
  \begin{equation} \label{eq1}
  \int_{\{ |\wt{\eta}_p | \ge \delta\}}  d\pmb{\eta_p}   d\pmb{\wt{\eta}_p}   \Phi(\pmb{\eta_p}, \pmb{\wt{\eta}_p})
    \xrightarrow{R\to\infty}  0.
    \end{equation}
 On the other hand,  from the  proofs of the estimates \eqref{F1} and \eqref{F6} in Lemma  \ref{LEM:H12}, we obtain  
 \begin{align*}
  \int_{\{ | \eta_r  -\wt{\eta}_r  |   <2\delta\}}  d\pmb{\eta_p}   d\pmb{\wt{\eta}_p}   \Phi(\pmb{\eta_p}, \pmb{\wt{\eta}_p})
  & \leq   
 C \sup_{\beta, \wt{\beta} \in \{ 0,1,2\}}   \int_{\R^2}\int_{[0,t]^2}d\eta_rd\wt{\eta}_rdw_rd\wt{w}_r\varphi(\eta_r)^{\beta}e^{-w_r\eta_r^2}\varphi(\wt{\eta})^{\wt{\beta}}e^{-\wt{w}_r\wt{\eta}_r^2} \\
 & \qquad \times  \1_{ \{  | \eta_r -\wt{\eta}_r  | < 2\delta\} }.
 \end{align*}
 Note also that for any $ \beta, \wt{\beta} \in \{ 0,1,2\}$,
 \[
 \int_{\R^2}\int_{[0,t]^2}d\eta_rd\wt{\eta}_rdw_rd\wt{w}_r\varphi(\eta_r)^{\beta}e^{-w_r\eta_r^2}\varphi(\wt{\eta}_r)^{\wt{\beta}}e^{-\wt{w}_r\wt{\eta}_r^2} \leq C.
 \]
Since $ {\bf 1}_{ \{  | \eta_r -\wt{\eta}_r  | < 2\delta\} }\to 0$ as $\delta\to 0$, we deduce
\begin{equation} \label{eq2}
 \lim_ {\delta \rightarrow 0} \limsup_{R\rightarrow  \infty} 
 \int_{\{ | \eta_r  -\wt{\eta}_r  |   <2\delta\}}  d\pmb{\eta_p}   d\pmb{\wt{\eta}_p}   \Phi(\pmb{\eta_p}, \pmb{\wt{\eta}_p})
 =0.
 \end{equation}
 Thus, \eqref{eq1} and \eqref{eq2} allow us to
 complete the proof of \eqref{ECU8}  in case {\bf (H1)}.

 \medskip

  \noindent{\textbf{Proof of \eqref{ECU8} under} $\mathbf{(H2)}$:} From  \eqref{F13-1}, the embedding result \eqref{F12} and Cauchy-Schwarz inequality, it follows that
  \begin{align*}
 \mathcal{Y}_{2,R,\delta}&\leq C\left\{\int_{[0,t]^{2p}} d\ell ^{2p} U_{1,R,\delta}^{\frac{1}{2H_0}}\right\}^{H_0}\left\{\int_{[0,t]^{2p}} d\ell^{2p}U_{2,R,\delta}^{\frac{1}{2H_0}}\right\}^{H_0}=:C\mathfrak{Z}_{1,R,\delta}^{H_0}\mathfrak{Z}_{2,R,\delta}^{H_0},
  \end{align*}
  where
  \begin{align*}
  \mathfrak{Z}_{1,R,\delta}&=\int_{[0,t]^{2p}}d\pmb{s_p}d\pmb{t_p}\Big(\int_{\R^p}\mu(d\pmb{\xi_p})\ell_R(\tau(\pmb{\xi_p}))|\F \mathfrak{f}|^2(\pmb{s_p},\pmb{\xi_p})\\
  & \qquad \times\int_{\{|\tau(\pmb{\xi_p})|<\delta\}}\mu(d\pmb{\xi_p})\ell_R(\tau(\pmb{\xi_p}))|\F \mathfrak{f}|^2(\pmb{t_p},\pmb{\xi_p})\Bigg)^{\frac{1}{2H_0}} \\
  \mathfrak{Z}_{2,R,\delta}&=\int_{[0,t]^{2p}}d\pmb{s_p}d\pmb{t_p}\Bigg(\int_{ \{ | a+b| < \delta\} \times\R^{p}}  \mu(d\pmb{\xi_r})  \mu(d\pmb{\eta_{p-r}})   \mu(d\pmb{\wt{\xi}_r})  \mu(d\pmb{\wt{\eta}_{p-r}})    \\
&   \qquad \times
  \big\vert\F \mathfrak{f}\big\vert^2 \big(  \pmb{s_p},\pmb{\eta_{p-r}}, \pmb{\wt{\xi}_r} \big)   \big\vert \F\mathfrak{f}\big\vert^2 \big(  \pmb{\wt{t}_p},   \pmb{\wt{ \eta}_{p-r}}, \pmb{\xi_r} \big)     \ell_R(\wt{a}+ b) \ell_R(a+\wt{b})\Bigg)^{\frac{1}{2H_0}}.
  \end{align*}
  For the term $\mathfrak{Z}_{1,R,\delta} $,  we deduce from \eqref{eqn-V2} that
  \begin{align*}
\mathfrak{Z}_{1,R,\delta}&\leq \left( \int_{[0,t]^{p}}d\pmb{s_p}\left(\int_{\R^p}\mu(d\pmb{\xi_p})\ell_R(\tau(\pmb{\xi_p}))|\F \mathfrak{f}|^2(\pmb{s_p},\pmb{\xi_p})\right)^{\frac{1}{2H_0}}\right)^2=T_{2,R}^2\leq C 
  \end{align*}
for all  $R>1$ and 
$ \delta>0$,    which shows the  uniform boundedness of $\mathfrak{Z}_{1,R,\delta}$ over $R>1$ and $\delta>0$.

Now let us consider the term $\mathfrak{Z}_{2,R,\delta}$. Similar to the analysis on the term $\mathcal{Y}_{1,R,\delta}$, one can get, with $\eta_0=\wt{\eta}_0=0$,
   \begin{align*}
\mathfrak{Z}_{2,R,\delta}  &=(p!)^{-\frac{2(1-H_0)}{H_0}}\int_{\text{SIM}_p(t)^2}d\pmb{w_p}d\pmb{\wt{w}_p}\Bigg(\int_{\R^{2p}}d\pmb{\eta_p}d\pmb{\wt{\eta}_p}\1_{\{|\eta_r+\wt{\eta}_p-\wt{\eta}_r|<\delta\}}\ell_R(\eta_p)\ell_R(\wt{\eta}_p)\\
&\qquad \times \prod_{j=1}^p|\eta_j-\eta_{j-1}|^{1-2H_1}e^{-w_j\eta_j^2} \prod_{j=1}^p|\wt{\eta}_j-\wt{\eta}_{j-1}|^{1-2H_1}e^{-\wt{w}_j\wt{\eta}_j^2}\Bigg)^{\frac{1}{2H_0}}\\
  &=(p!)^{-\frac{2(1-H_0)}{H_0}}\int_{\text{SIM}_p(t)^2}d\pmb{w_p}d\pmb{\wt{w}_p}\Bigg(\int_{\R^{2p}}d\pmb{\eta_p}d\pmb{\wt{\eta}_p}\1_{\{|\eta_r+\wt{\eta}_p-\wt{\eta}_r|<\delta\}}\ell_R(\eta_p)\ell_R(\wt{\eta}_p)\\
  & \qquad \times \prod_{j=1}^p|\eta_j-\eta_{j-1}|^{1-2H_1}e^{-w_j\eta_j^2} \left(  \1_{\{| \wt{\eta}_p|<\delta\}}+ \1_{\{| \wt{\eta}_p|\geq\delta\}}\right)\prod_{j=1}^p|\wt{\eta}_j-\wt{\eta}_{j-1}|^{1-2H_1}e^{-\wt{w}_j\wt{\eta}_j^2}\Bigg)^{\frac{1}{2H_0}}\\
  &=:\mathfrak{Z}_{21,R,\delta}+\mathfrak{Z}_{22,R,\delta},
   \end{align*}
   where 
 \begin{align*}
\mathfrak{Z}_{21,R,\delta}&= (p!)^{-\frac{2(1-H_0)}{H_0}} \int_{\text{SIM}_p(t)^2}d\pmb{w_p}d\pmb{\wt{w}_p}\Bigg(\int_{\R^{2p}}d\pmb{\eta_p}d\pmb{\wt{\eta}_p}\1_{\{|\eta_r+\wt{\eta}_p-\wt{\eta}_r|<\delta\}}\1_{\{| \wt{\eta}_p|<\delta\}}\ell_R(\eta_p)\ell_R(\wt{\eta}_p)\\
  &\qquad\qquad\qquad\qquad\qquad \times\prod_{j=1}^p|\eta_j-\eta_{j-1}|^{1-2H_1}e^{-w_j\eta_j^2}\prod_{j=1}^p|\wt{\eta}_j-\wt{\eta}_{j-1}|^{1-2H_1}e^{-\wt{w}_j\wt{\eta}_j^2}\Bigg)^{\frac{1}{2H_0}}
 \end{align*}
 and 
 \begin{align*}
\mathfrak{Z}_{22,R,\delta}&= (p!)^{-\frac{2(1-H_0)}{H_0}}  \int_{\text{SIM}_p(t)^2}d\pmb{w_p}d\pmb{\wt{w}_p}\Bigg(\int_{\R^{2p}}d\pmb{\eta_p}d\pmb{\wt{\eta}_p}\1_{\{|\eta_r+\wt{\eta}_p-\wt{\eta}_r|<\delta\}}\1_{\{| \wt{\eta}_p|\geq\delta\}}\ell_R(\eta_p)\ell_R(\wt{\eta}_p)\\
  &\qquad\qquad\qquad\qquad\qquad \times\prod_{j=1}^p|\eta_j-\eta_{j-1}|^{1-2H_1}e^{-w_j\eta_j^2}\prod_{j=1}^p|\wt{\eta}_j-\wt{\eta}_{j-1}|^{1-2H_1}e^{-\wt{w}_j\wt{\eta}_j^2}\Bigg)^{\frac{1}{2H_0}}.
 \end{align*}
For $\mathfrak{Z}_{21,R,\delta}$, we have
 \begin{align*}
\mathfrak{Z}_{21,R,\delta}&\leq C  \sum_{\pmb{\beta},\, \pmb{\wt{\beta}}\in\mathcal{A}_p}\int_{\text{SIM}_p(t)^2}d\pmb{w_p}d\pmb{\wt{w}_p}\Bigg(\int_{\R^{2p}}d\pmb{\eta_p}d\pmb{\wt{\eta}_p}\1_{\{|\eta_r-\wt{\eta}_r|<2\delta\}}  \ell_R(\eta_p)\ell_R(\wt{\eta}_p)\\
  &\qquad  \times\prod_{j=1}^p|\eta_j|^{\beta_j(1-2H_1)}e^{-w_j\eta_j^2}\prod_{j=1}^p|\wt{\eta}_j|^{\wt{\beta}_j(1-2H_1)}e^{-\wt{w}_j\wt{\eta}_j^2}\Bigg)^{\frac{1}{2H_0}}\\
  &\leq C  \sum_{\pmb{\beta},\, \pmb{\wt{\beta}}\in\mathcal{A}_p}\Bigg[
  \left(
  \int_{[0,t]}dw_p\left(\int_{\R}d\eta_p\ell_R(\eta_p)|\eta_p|^{\beta_p(1-2H_1)}e^{-w_p\eta_p^2}\right)^{\frac{1}{2H_0}}\right)^2\\
  &\qquad \times\left(\int_{\text{SIM}_{p-2}(t)}d\pmb{w_{p-2}}\left(\int_{\R^{(p-2)}}d\pmb{\eta_{p-2}}\prod_{j\in \{1, \dots, p-1\}, j\neq r}|\eta_j|^{\beta_j(1-2H_1)}e^{-w_j\eta_j^2}\right)^{\frac{1}{2H_0}}\right)^2\\
  &\qquad\times \int_{[0,t]^2}dw_rd\wt{w}_r\left(\int_{\R^2}d\eta_rd\wt{\eta}_r\1_{\{|\eta_r- \wt{\eta}_r|<2\delta\}}|\eta_r|^{\beta_r(1-2H_1)}|\wt{\eta}_r|^{\wt{\beta}_r(1-2H_1)}e^{-w_r\eta_r^2}e^{-\wt{w}_r\wt{\eta}_r^2}\right)^{\frac{1}{2H_0}}\Bigg].
 \end{align*}
 Using some previous calculations, we  obtain the following bound
\[
 \sup_{R>1} \mathfrak{Z}_{21,R,\delta} \le
 C
 \int_{[0,t]^2}dvds  \left(\int_{\R^2}dx dy \1_{\{|x-y|<2\delta\}}|x|^{\beta_r(1-2H_1)}|y|^{\wt{\beta}_r(1-2H_1)}e^{-vx ^2}e^{- sy ^2}\right)^{\frac{1}{2H_0}}
 \]
 where $\beta_r, \wt{\beta}_r \in \{ 0,1,2\}$ and this proves, by the dominated convergence theorem, that
 \[
 \lim_{\delta \rightarrow 0} \limsup_{R>1} \mathfrak{Z}_{21,R,\delta}=0.
 \]
 On the other hand, the term  $\mathfrak{Z}_{22,R,\delta}$ can be treated as $T_{1,R,\delta}$ and we obtain  
 \[
 \lim_{R\rightarrow \infty} \mathfrak{Z}_{22,R,\delta}=0,~\text{for any fixed $\delta>0$}.
 \]
Hence, the  proof of \eqref{ECU8} under {\bf (H2)} is completed and this ends the proof of Proposition  \ref{prop2}.
\end{proof}

 \section{Proof of tightness} \label{SEC5}
 
 In this section, we will give the proof of the tightness by following the  strategy proposed in \cite[Section 3.3]{NZ19BM}. 
 \begin{proposition}
For any fixed $T>0$, any $0< s< t \leq T \leq R$ and any  integer $k\geq 2$ 
  \begin{align}\label{G00}
  \frac{1}{\sqrt{R}}\big\|  A_t(R) - A_s(R)\big\|_{k}\leq C  \vert t-s\vert^{1/2},
  \end{align}
  where $C=C_{T,k}$ is a constant that depends  on $T$ and  $k$. 
 \end{proposition}
\begin{proof} 
From the definition of mild solution (see Definition \ref{sko-sol}), we can write
\[
u(t,x) = 1 + \int_{\R_+\times\R} G(t-s_1, x- y_1){\bf 1}_{[0,t)}(s_1) u(s_1, y_1) W(ds_1, dy_1).
\]
Therefore,    for $s<t$  we have the decomposition of $u(t,x) - u(s,x)$ into
\begin{align*}
   \int_{\R_+\times\R}  d_1(s,t,x;s_1, y_1)  u(s_1, y_1) W(ds_1, dy_1)  +  \int_{\R_+\times\R}  d_2(s,t,x;s_1, y_1)  u(s_1, y_1) W(ds_1, dy_1), 
\end{align*}
with
$d_1(s,t,x;s_1, y_1)  =  {\bf 1}_{[0,s)}(s_1) \big[ G(t-s_1, x- y_1)  -  G(s-s_1, x- y_1) \big] 
$ and
\begin{equation*}
 d_2(s,t,x;s_1, y_1) =  {\bf 1}_{[s,t)}(s_1) G(t-s_1, x- y_1). 
\end{equation*}
 Now we   express $A_t(R)-A_s(R)$ as a sum of two chaos expansions that correspond to $d_1$ and $d_2$:
\[
A_t(R)-A_s(R)  =  \sum_{p= 1}^\infty J_{1,p,R}   +   \sum_{q= 1}^\infty J_{2,q,R},
\]
 where $J_{i,p,R} =\int_{B_R}  I_p^W\big( \mathfrak{g}_{i,p,x} \big) dx  $ for $i\in\{1,2\}$ and 
\begin{align*}
\mathfrak{g}_{1,p,x}(\pmb{s_p},\pmb{y_p}) &= \frac{1 }{p!} \sum_{\sigma\in\mathfrak{S}_p} {\bf 1}_{\Delta_p(s)}(\pmb{s^\sigma_p})  d_1(s,t,x;s_{\sigma(1)}, y_{\sigma(1)})  \prod_{j=1}^{p-1}  G(s_{\sigma(j)} -s_{\sigma(j+1)}, y_{\sigma(j)}- y_{\sigma(j+1)}),\\
\mathfrak{g}_{2,p,x}(\pmb{s_p},\pmb{y_p})  &= \frac{1}{p!} \sum_{\sigma\in\mathfrak{S}_p} {\bf 1}_{\Delta_p(s,t)}(\pmb{s^\sigma_p})  G(t-s_{\sigma(1)}, x-y_{\sigma(1)}) \prod_{j=1}^{p-1}  G(s_{\sigma(j)} -s_{\sigma(j+1)}, y_{\sigma(j)}- y_{\sigma(j+1)}),
\end{align*}
 with $\Delta_p(s,t) = \{ t > s_1 > \dots > s_p > s\}$.  
 Finally, we  apply \eqref{HYPER} to get
\begin{align*}
\frac{1}{\sqrt{R}}\big\| A_t(R)-A_s(R)   \big\|_{k} &\leq  \frac{1}{\sqrt{R}}\sum_{p= 1}^\infty \Big(  \big\| J_{1,p,R}  \big\|_{k} +  \big\| J_{2,p,R}  \big\|_{k} \Big)\\
& \leq  \frac{1}{\sqrt{R}}\sum_{p= 1}^\infty  (k-1)^{p/2} \Big(  \big\| J_{1,p,R}  \big\|_{2} +  \big\| J_{2,p,R}  \big\|_{2} \Big).
\end{align*}
Now, let us estimate $ \big\| J_{i,p,R}  \big\|_{2}$ for $i\in\{1,2\}$ and $p\in\{1,2,\dots\}$.

\medskip
\noindent
{\bf Case  $i=2$,  $p\ge 2$}:  Following the arguments in  \cite[Section 3.3]{NZ19BM}, we write
\begin{align*}
 \big\| J_{2,p,R}  \big\|_2^2 &=p! \int_{-R}^R\int_{-R}^R  dxdy \big\langle \mathfrak{g}_{2,p,x},  \mathfrak{g}_{2,p,y} \big\rangle_{\H^{\otimes p}} = \frac{1}{p!} \int_{-R}^R\int_{-R}^R dxdy \E\big[\big(\mathcal{I}^{1,2}_{t-s,t-s}(x-y)\big)^p \big] \\
 &\leq \frac{2 R}{p!} \int_\R dz \E\big[  \big|\mathcal{I}^{1,2}_{t-s,t-s}(z)\big|^p \big].
 \end{align*} 
From \eqref{F15} and \eqref{view331},  in view of the factor $t$ in the right-hand side of \eqref{F15}  and the factor
$t^{ 4H_0 \mathfrak{h} +\frac 12}$ in the right-hand side of \eqref{view331}, with  $4H_0 \mathfrak{h} +\frac 12  >1$,  it follows easily that
\[
 \sum_{p= 2}^\infty \left(\frac{1}{p!} \int_\R dz   \E\big[ |  \mathcal{I}^{1,2}_{t-s,t-s}(z)|^p \big] \right)^{1/2} (k-1)^{p/2} \leq C (t-s)^{1/2}.
\]
Hence, ${\displaystyle
\frac{1}{\sqrt{R}} \sum_{p=2}^\infty   \big\| J_{2,p,R}  \big\|_{k}  \leq C (t-s)^{1/2}.
 }$

\medskip
\noindent
{\bf Case  $i=2$,  $p=1$}: We have $\mathfrak{g}_{2,1,x}(s_1, y_1) =\1_{[s,t]}(s_1)  G(t-s_1, x-y_1)$, so it follows from \eqref{def:LR} that
\begin{align*}
\frac{1}{4\pi R} \big\| J_{2,1,R} \big\|_2^2 & =\frac{1}{4\pi R}  \int_s^t\int_s^t drdv \gamma_0(r-v) \int_\R d\xi \int_{-R}^R\int_{-R}^R dxdy e^{-\i (x-y)\xi} \varphi(\xi) e^{-\frac{t-r+t-v}{2}\xi^2} \\
&= \int_s^t\int_s^t drdv \gamma_0(r-v) \int_\R d\xi \ell_R(\xi) \varphi(\xi) e^{-\frac{t-r+t-v}{2}\xi^2}.
\end{align*}
Then, from \eqref{uni-bound} we obtain
\[
\frac{1}{4\pi R} \big\| J_{2,1,R} \big\|_2^2 \leq (t-s) \Gamma_T   \int_\R d\xi \ell_R(\xi) \varphi(\xi)\leq C(t-s).
\]
 As a consequence, 
\[
\frac{1}{\sqrt{R}} \big\| J_{2,1,R} \big\|_2 \leq C \sqrt{t-s} ~\text{so that}~ \frac{1}{\sqrt{R}} \big\| J_{2,1,R} \big\|_k \leq C \sqrt{t-s}.
\]
Therefore, we have proved 
\begin{equation}\label{eqn-J2}
\frac{1}{\sqrt{R}} \sum_{p=1}^\infty   \big\| J_{2,p,R}  \big\|_{k}  \leq C (t-s)^{1/2}.
 \end{equation}

Next, we shall prove
\begin{equation}\label{eqn-J1}
\frac{1}{\sqrt{R}} \sum_{p=1}^\infty   \big\| J_{1,p,R}  \big\|_{k}  \leq C (t-s)^{1/2}.
 \end{equation}
 
 \medskip
\noindent
{\bf Case  $i=1$,  $p=1$}:
We have $\mathfrak{g}_{1,1,x}(s_1, y_1) =\1_{\{ s_1 < s\} }  \big[ G(t-s_1, x-y_1) - G(s-s_1, x-y_1)\big]$ and then,
\begin{align*}
\frac{1}{4\pi R} \big\| J_{1,1,R} \big\|_2^2   & = \int_0^s\int_0^s drdv\gamma_0(r-v) \int_\R d\xi \varphi(\xi) \ell_R(\xi) \left(e^{-\frac{t-r}{2} \xi^2 }  -e^{-\frac{s-r}{2} \xi^2 } \right)\left(e^{-\frac{t-v}{2} \xi^2 }  -e^{-\frac{s-v}{2} \xi^2 } \right) \\
&\leq  \int_0^s\int_0^s drdv\gamma_0(r-v) \int_\R d\xi \varphi(\xi) \ell_R(\xi)  e^{-\frac{s-r}{2} \xi^2 }  \left(1 - e^{-\frac{t-s}{2} \xi^2 }  \right)  \\
&\leq \frac{t-s}{2}  \int_0^s\int_0^s drdv\gamma_0(r-v) \int_\R d\xi \varphi(\xi) \ell_R(\xi)  e^{-\frac{s-r}{2} \xi^2 } \xi^2 \quad\text{as $1-e^{-x}\leq x, \forall x\geq 0$} \\
&\leq \frac{t-s}{2} \Gamma_s \int_\R d\xi \varphi(\xi) \ell_R(\xi)\xi^2 \int_0^s dr  e^{-\frac{s-r}{2} \xi^2 }\\
& =(t-s) \Gamma_s \int_\R d\xi \varphi(\xi) \ell_R(\xi)\xi^2 \frac{1 - e^{-\frac{s\xi^2}{2}  }}{\xi^2} \\
&\leq (t-s) \Gamma_T \int_\R d\xi \varphi(\xi) \ell_R(\xi)  \leq C (t-s),
\end{align*}
for $R\geq T$, in view of \eqref{uni-bound}. This implies 
\[
\frac{1}{\sqrt{R}} \big\| J_{1,1,R} \big\|_2 \leq C \sqrt{t-s} ~\text{so that}~ \frac{1}{\sqrt{R}} \big\| J_{1,1,R} \big\|_k \leq C \sqrt{t-s}.
\]

 \medskip
\noindent
{\bf Case  $i=1$,  $p\ge 2$}: Let us first compute the Fourier transform of $\mathfrak{g}_{1,p,x}$
\begin{align*}
\big(\F \mathfrak{g}_{1,p,x}\big)(\pmb{s_p}, \pmb{\xi_p}  ) &= \frac{1}{p!} \sum_{\sigma\in\mathfrak{S}_p} \1_{ \Delta_p(s)} (\pmb{s^\sigma_p}) \int_{\R^p} d\pmb{y_p} e^{-\i \pmb{\xi_p} \cdot \pmb{y_p}} \\
& \qquad \times \big[ G(t-s_{\sigma(1)}, x-y_{\sigma(1)}) - G(s-s_{\sigma(1)}, x-y_{\sigma(1)})\big] \\
& \qquad     \times \prod_{j=1}^{p-1}  G(s_{\sigma(j)} -s_{\sigma(j+1)}, y_{\sigma(j)}- y_{\sigma(j+1)}) \\
& =e^{-\i \tau(\pmb{\xi_p} ) x}  \frac{1}{p!}  \1_{[0,s]}(\pmb{s_p}) \left(  \E\left[  e^{-\i \sum_{j=1}^p (B_t - B_{s_j} ) \xi_j  } \right] -\E\left[  e^{-\i \sum_{j=1}^p (B_s - B_{s_j} ) \xi_j  } \right]  \right).
\end{align*}
Therefore, 
\begin{align}
&\quad \frac{1}{4\pi R} \big\| J_{1,p,R} \big\|_2^2 \notag \\
&= \frac{p!}{4\pi R} \int_{-R}^R\int_{-R}^R  dxdy \big\langle  \mathfrak{g}_{1,p,x},  \mathfrak{g}_{1,p,y} \big\rangle_{\H^{\otimes p}}   \notag  \\
&= \frac{p!}{4\pi R} \int_{-R}^R\int_{-R}^R  dxdy  \int_{[0,s]^{2p}} d\pmb{s_p}  d\pmb{r_p} \prod_{i=1}^p \gamma_0(s_i-r_i)   \int_{\R^p} \mu(d\pmb{\xi_p}) \big(\F \mathfrak{g}_{1,p,x}\big)(\pmb{s_p}, \pmb{\xi_p}  )\big(\F \mathfrak{g}_{1,p,y}\big)(\pmb{r_p}, -\pmb{\xi_p}  )   \notag   \\
&=  \frac{1}{p!}   \int_{[0,s]^{2p}}    d\pmb{s_p}  d\pmb{r_p}  \left(  \prod_{i=1}^p \gamma_0(s_i-r_i)  \right) \int_{\R^p} \mu(d\pmb{\xi_p}) \ell_R\big(     \tau(\pmb{\xi_p} ) \big)  \Bigg(  \E\left[  e^{-\i \sum_{j=1}^p (B_t - B_{s_j} ) \xi_j  } \right]  \notag \\
 &\qquad\quad -\E\left[  e^{-\i \sum_{j=1}^p (B_s - B_{s_j} ) \xi_j  } \right]  \Bigg)  \Bigg(  \E\left[  e^{-\i \sum_{j=1}^p (B_t - B_{r_j} ) \xi_j  } \right] -\E\left[  e^{-\i \sum_{j=1}^p (B_s - B_{r_j} ) \xi_j  } \right]  \Bigg). \label{eq:continue}
\end{align} 
Let us consider separately case $(\mathbf{H1})$ and  case $(\mathbf{H2})$.

\medskip

In case $(\mathbf{H1})$, using the local integrability of $\gamma_0$ and the fact
\[
\E\left[  e^{-\i \sum_{j=1}^p (B_s - B_{r_j} ) \xi_j  } \right]  -  \E\left[  e^{-\i \sum_{j=1}^p (B_t - B_{r_j} ) \xi_j  } \right]  \in [0,1],
\]
we can write 
\begin{align*}
\frac{1}{4\pi R} \big\| J_{1,p,R} \big\|_2^2  & \leq \frac{\Gamma_s^p}{ p! } \int_{[0,s]^p} d\pmb{s_p}\int_{\R^p} \mu(d\pmb{\xi_p}) \ell_R\big(     \tau(\pmb{\xi_p} ) \big) \\
& \qquad \times   \left(  \E\left[  e^{-\i \sum_{j=1}^p (B_s - B_{s_j} ) \xi_j  } \right]  -\E\left[  e^{-\i \sum_{j=1}^p (B_t - B_{s_j} ) \xi_j  } \right]  \right). \\
\end{align*}
We have
 \begin{align}   \notag
0&  \leq  \E\left[  e^{-\i \sum_{j=1}^p (B_s - B_{s_j} ) \xi_j  } \right]  -\E\left[  e^{-\i \sum_{j=1}^p (B_t - B_{s_j} ) \xi_j  } \right]  \\ \notag
& =
\E \left[  e^{-\i \sum_{j=1}^p (B_s - B_{s_j} ) \xi_j  }   \left( 1- e^{-\i  (B_t-B_s) \tau(\pmb{ \xi_p})} \right) \right]   \\
&=  \E\left[  e^{-\i \sum_{j=1}^p (B_s - B_{s_j} ) \xi_j  } \right] \left(1- e^{ -\frac 12 (t-s) \tau(\pmb{\xi_p}) ^2} \right) \notag \\ 
& \leq  \frac{t-s}{2}  \tau(\pmb{\xi_p} ) ^2  \E\left[  e^{-\i \sum_{j=1}^p (B_s - B_{s_j} ) \xi_j  } \right].  \label{F9}
  \end{align}
As a consequence, for all $R>T$, we can write
 \begin{align*}
 \frac{1}{4\pi R} \big\| J_{1,p,R} \big\|_2^2  &\leq \frac{ (t-s)\Gamma_s^p}{2p!} \int_{[0,s]^p} d\pmb{s_p}\int_{\R^p} \mu(d\pmb{\xi_p}) \ell_R\big(     \tau(\pmb{\xi_p} ) \big)   \tau(\pmb{\xi_p} ) ^2   \E\left[  e^{-\i \sum_{j=1}^p (B_s - B_{s_j} ) \xi_j  } \right]    \\
 & \leq  \frac{(t-s)\Gamma_T^p }{2 p!} \int_{[0,s]^p} d\pmb{s_p}\int_{\R^p} \mu(d\pmb{\xi_p}) \frac{ \sin^2(R   \tau(\pmb{\xi_p} )  ) }{\pi R}  \E\left[  e^{-\i \sum_{j=1}^p (B_s - B_{s_j} ) \xi_j  } \right] \\
 &\leq  \frac{ (t-s)\Gamma_T^p}{2\pi T  p!} \int_{[0,s]^p} d\pmb{s_p}\int_{\R^p} \mu(d\pmb{\xi_p})    \E\left[  e^{-\i \sum_{j=1}^p B_{s_j}  \xi_j  } \right]\\
 &=  \frac{ (t-s)\Gamma_T^p}{2\pi T  p!}  K_{1,p} (1,s).
 \end{align*}
  Then, applying  inequality  \eqref{F1}, yields, for $R>T$,
   \begin{align*}
 \frac{1}{\sqrt{R}} \sum_{p= 2}^\infty \big\| J_{1,p,R} \big\|_k &\leq  \frac{1}{\sqrt{R}} \sum_{p= 2}^\infty  \big\| J_{1,p,R} \big\|_2 (k-1)^{p/2} \\
 & \le \sqrt{\frac 1{T \kappa_0}}  (t-s)^{1/2} \sum_{p= 2}^\infty \sqrt{     (8\kappa_0C_N (k-1)\Gamma_T)^p    \exp\Big( \frac{s D_N}{2C_N} \Big)}
  \leq C(t-s)^{1/2},
 \end{align*}
 for some large $N$ such that $8\kappa_0C_N' (k-1)\Gamma_T<1$.  This, together with the previous estimate for $p=1$, leads to  
 \eqref{eqn-J1} under the hypothesis ($\mathbf{H1}$).
 
 \medskip
 
Consider now the  case $(\mathbf{H2})$.   Let us continue with \eqref{eq:continue} and apply the embedding result \eqref{F12}  and Cauchy-Schwarz inequality to write
 \begin{align*}
 \frac{1}{4\pi R} \big\| J_{1,p,R} \big\|_2^2 
&   \leq \frac{ C_{H_0}^p}{p!} \Bigg[ \int_{[0,s]^p} d\pmb{s_p}   \Big[  \int_{\R^p} \mu(d\pmb{\xi_p}) \ell_R\big(     \tau(\pmb{\xi_p} ) \big)   \\
& \qquad   \times  \left\{  \E\left[  e^{-\i \sum_{j=1}^p (B_s- B_{r_j} ) \xi_j  } \right] -\E\left[  e^{-\i \sum_{j=1}^p (B_t - B_{r_j} ) \xi_j  } \right]  \right\}   \Big]^{\frac{1}{2H_0}}  \Bigg]^{2H_0}.
\end{align*} 
We deduce from \eqref{F9} and $\ell_R(x) = \frac{\sin^2(Rx)}{\pi Rx^2}$ that
\begin{align*}
   \frac{1}{4\pi R} \big\| J_{1,p,R} \big\|_2^2 
 & \leq  \frac{C_{H_0}^p}{p!}    \frac {t-s} { 2\pi T}  \left[  \int_{[0,s]^p} d\pmb{s_p} \left( \int_{\R^p} \mu(d\pmb{\xi_p})  \E\left[  e^{-\i \sum_{j=1}^p (B_s- B_{s_j} ) \xi_j  } \right]  \right)^{\frac{1}{2H_0}} \right]^{2H_0}\\
 &= \frac{C_{H_0}^p}{p!}    \frac {t-s} { 2\pi T}   [K_{2,p} (1,s) ]^{2H_0}.
\end{align*}
 Then, applying the estimate \eqref{F8} yields
 \begin{align*}
 \frac{1}{R} \big\| J_{1,p,R} \big\|_2^2 &  \leq \frac{2(t-s)}{T} (p!)^{2H_0-1}  \big(C_{H_0}C_2^{2H_0} \big)^p \left(  \frac{T^{ p\mathfrak{h}} }{\Gamma(p \mathfrak{h}+1)}  \right)^{2H_0}, 
 \end{align*}
 for $s<t\leq T \leq R$. Hence,  for $s<t\leq T \leq R$,
  \begin{align}  \notag
 \frac{1}{\sqrt{R}} \sum_{p=2}^\infty \big\| J_{1,p,R} \big\|_k &\leq  \frac{1}{\sqrt{R}} \sum_{p= 2}^\infty \big\| J_{1,p,R} \big\|_2 (k-1)^{p/2} \\ \notag
&\leq \sqrt{\frac 2T}  (t-s)^{1/2} \sum_{p=2}^\infty \big(C_{H_0}C_2^{2H_0}\big)^{p/2} (p!)^{\frac{2H_0-1}{2}}  \left(  \frac{T^{  p \mathfrak{h} }}{\Gamma(p \mathfrak{h}+1)}     \right)^{H_0} \\  \label{F19}
& \leq C (t-s)^{1/2},
 \end{align}
 where the last inequality follows from  the fact that $\frac {2H_0-1}2 - \mathfrak{h} H_0 = -\frac {H_1} 2$.  From inequality \eqref{F19}  together with the one for $p=1$,  we can obtain \eqref{eqn-J1} in case {\bf (H2)}.

 \medskip
 
 Therefore, in both cases {\bf (H1)} and {\bf (H2)}, the estimates \eqref{eqn-J2} and \eqref{eqn-J1} yield  \eqref{G00}, and hence   the proof of tightness is concluded. 
 \end{proof}

\section{Proof of Proposition \ref{prop1}}\label{FKproof}

 For $a>0$, we define 
 $
 \varphi_a(t) = \frac{1}{a}  \1_{[0,a]}(t)$,  $t\in\R$,
 so $\big\{\varphi_a(\ast) G(\e,\cdot): \e,a>0 \big\}$  is an approximation to the delta function.   Consider the approximating equation 
 \begin{align}\label{appeq}
 \frac{ \partial u^{\e, a}} {\partial t}(t,x) = \frac{1}{2}\Delta u^{\e, a}(t,x)  +  u^{\e, a}(t,x) \diamond  W^{\e,a}_{t,x},
 \end{align}
 with initial condition $u^{\e, a}(t,x) =1$ and 
 \[
 W^{\e,a}_{t,x} = \int_0^t \int_\R \varphi_a(t-s) G(\e,x-y) W(ds,dy). 
 \]
   By following exactly the same lines as in the proof of \cite[Proposition 5.2]{HN09}, we can  
  show that 
  \begin{equation} \label{ECU9}
  u^{\e, a}(t,x) = \E_B\Big[  \exp\Big(  W(A^{\e, a, B}_{t,x}) - \dfrac{1}{2} \| A^{\e, a,B}_{t,x}\|^2_{\mathfrak{H}}   \Big) \Big]
  \end{equation}
   solves  equation \eqref{appeq}, where $ \E_B$ denotes the expectation with respect to the randomness $B$ and 
 \begin{align*}
 A^{\e, a, B}_{t,x}(r,y)  &= \int_0^t \varphi_a(t-s-r) G(\e, B_{s} +x-y)ds\, \1_{[0,t]}(r)  \\
 &= \frac{1}{a} \int_0^{a\wedge (t-r)}  G(\e, B_{t-r-s} +x-y)ds\, \1_{[0,t]}(r),
 \end{align*}
 with $\{B_s, s\in\R_+\}$ a standard   Brownian motion on $\R$ independent of $W$. We omit the details of the proof of \eqref{ECU9}

  The proof of the Feynman-Kac formula begins with the following identity:    For any $t_1, \dots, t_k$ and any $x_1, \dots x_k$ with $k\geq 2$, 
  \begin{align} \label{quan1}
 \E\left(\prod_{j=1}^k u^{\e,a}(t_j, x_j) \right)= \E\left[ \exp\left( \sum_{1\leq i < j \leq k} \big\langle A^{\e, a, B^i}_{t_i,x_i}, A^{\e, a, B^j}_{t_j, x_j}   \rangle_\H \right) \right],
 \end{align}
  where $B^1, \dots, B^k$ are $k$ i.i.d copies of $B$; see \emph{e.g.} equation (3.22) in \cite[page 15]{HHNT15}.
  
  The remaining proof consists of three steps. In Step 1, we will prove that  the expectations in \eqref{quan1} are uniformly bounded over $\e>0, a\in (0, c_0)$ for some $c_0>0$; the second step is devoted to proving  that
  \[
   \big\langle A^{\e, a, B^i}_{t_i,x_i}, A^{\e, a, B^j}_{t_j, x_j}   \rangle_\H \xrightarrow[a. s.]{a\downarrow 0} \mathcal{I}^{i,j}_{t_i, t_j,\e}(x_i-x_j) \xrightarrow[L^p(\Omega)]{\e\downarrow 0} \mathcal{I}_{t_i, t_j}^{i,j}(x_i- x_j), 
  \]
  where  $ \mathcal{I}^{i,j}_{t_i, t_j,\e}(x_i- x_j)$ and $ \mathcal{I}^{i,j}_{t_i, t_j}(x_i-x_j)$ are respective limits whose expressions will be clear later; in Step 3, we will show the $L^p(\Omega)$-convergence of $u^{\e, a}(t,x)$ to $u(t,x)$ as $\e, a\downarrow 0$.   Combining these steps yields
  \[
   \E\left(\prod_{j=1}^k u(t_j, x_j) \right) = \E\left[ \exp\left( \sum_{1\leq i < j \leq k}    \mathcal{I}^{i,j}_{t_i, t_j}(x_i-x_j) \right) \right],
   \]
  which is   formula \eqref{prop:FK}.  \\

 \noindent{\bf Step 1}:  It suffices to show  that for $i<j$ and any $\lambda>0$,
    \begin{align}\label{expm}
\sup\left\{  \E\left[ \exp\left( \lambda  \Big\vert  \big\langle A^{\e, a, B^i}_{t_i,x_i}, A^{\e, a, B^j}_{t_j, x_j}   \big\rangle_\H\Big\vert  \right) \right] : a\in(0,t_i\wedge t_j], \e >0 \right\} < +\infty.
    \end{align} 
 Similar to (3.23) and (3.24) in  \cite[page 15]{HHNT15}, we have 
\begin{align*}
  \big\langle A^{\e, a, B^i}_{t_i,x_i}, A^{\e, a, B^j}_{t_j, x_j}   \big\rangle_\H & =\int_\R  \mu(d \xi) e^{-\e \xi^2} e^{-\i \xi(x_i-x_j)} \int_0^{t_i}  \int_0^{t_j} drd\wt{r} \gamma_0(t_i-r-t_j+\wt{r} ) \\
  &\qquad \times \frac{1}{a^2} \int_0^{a\wedge r} \int_0^{a\wedge \wt{r}} ds d\tilde{s}   e^{-\i \xi (B^i_{r-s} - B_{\wt{r}-\tilde{s}}^j)}.
  \end{align*}
Thus, with $m=2n\geq 2$, we can write
\begin{align}
\E\Big[\big\langle A^{\e, a, B^i}_{t_i,x_i}, A^{\e, a, B^j}_{t_j, x_j}   \big\rangle_\H^m \Big]  &= \int_{\R^m} \mu(d\pmb{\xi_m}) \left(\prod_{k=1}^m e^{-\e\xi^2_k} e^{-\i \xi_k(x_i-x_j)}  \right) \int_{[0,t_i]^m\times [0, t_j]^m} d\pmb{r_m} d\pmb{\wt{r}_m} \notag \\
& \quad \times \left( \prod_{k=1}^m \gamma_0(t_i-r_k-t_j+\wt{r}_k)\right) \frac{1}{a^{2m}}\int_{[0,a]^{2m}} d\pmb{s_m} d\pmb{\wt{s}_m} \1_{\{ s_k \leq r_k , \wt{s}_k\leq \wt{r}_k ; \forall k\}}\notag \\
&\quad \times \E\left[  \exp\left(-\i \sum_{k=1}^m\xi_k (B^i_{ r_k -s_k} - B_{ \wt{r_k}-\wt{s}_k}^j) \right) \right]  \notag\\
&\leq    \frac{1}{a^{2m}}\int_{[0,a]^{2m}} d\pmb{s_m} d\pmb{\wt{s}_m}\int_{\R^m} \mu(d\pmb{\xi_m})  \int_{[0,t_i]^m\times [0, t_j]^m} d\pmb{r_m} d\pmb{\wt{r}_m} \1_{\{ s_k \leq r_k , \wt{s}_k\leq \wt{r}_k ; \forall k\}}  \notag \\
& \quad    \times \left( \prod_{k=1}^m \gamma_0(t_i-r_k-t_j+\wt{r}_k)\right)  \E\left[  \exp\left(-\i \sum_{k=1}^m\xi_k (B^i_{ r_k -s_k} - B_{ \wt{r_k}-\wt{s}_k}^j) \right) \right] \label{begin}
\end{align}
and for $s_k \leq r_k , \wt{s}_k\leq \wt{r}_k$,
\begin{align*}
& \E\left[  \exp\left(-\i \sum_{k=1}^m\xi_k (B^i_{ r_k -s_k} - B_{\wt{r_k}-\wt{s}_k}^j) \right) \right] \\
 = &\ \exp\left( -\frac{1}{2} {\rm Var} \sum_{k=1}^m \xi_k (B^i_{ r_k -s_k} - B_{ \wt{r_k}-\wt{s}_k}^j) \right) \leq  \exp\left( -\frac{1}{2} {\rm Var} \sum_{k=1}^m \xi_k  B^i_{ r_k -s_k}\right) .
 \end{align*}

 Now let us consider  case {\bf (H1)} where $\gamma_0$ is locally integrable and the spectral density $\varphi$ satisfies the modified Dalang's condition \eqref{mDc}. 
 By changing $(r_k-s_k, \wt{r}_k - \wt{s}_k)$ to $(r_k,\wt{r}_k)$ for each $k$, we get the following estimate   $(a\leq t_i\wedge t_j)$
\begin{align*}
\E\Big[\big\langle A^{\e, a, B^i}_{t_i,x_i}, A^{\e, a, B^j}_{t_j, x_j}   \big\rangle_\H^m \Big] & \leq   \frac{1}{a^{2m}}\int_{[0,a]^{2m}} d\pmb{s_m} d\pmb{\wt{s}_m}\int_{\R^m} \mu(d\pmb{\xi_m})  \int_{[0,t_i]^m\times [0, t_j]^m} d\pmb{r_m} d\pmb{\wt{r}_m}    \\
& \qquad\qquad   \times \left( \prod_{k=1}^m \gamma_0(t_i-r_k - s_k-t_j+\wt{r}_k + \wt{s}_k)\right)  \exp\left( -\frac{1}{2} {\rm Var} \sum_{k=1}^m \xi_k B^i_{ r_k }   \right) \\
 &\le \Gamma^m_{t_i+t_j} \int_{\R^m} \mu(d\pmb{\xi_m})  \int_{[0,t_i]^m} d\pmb{r_m}       \exp\left( -\frac{1}{2} {\rm Var} \sum_{k=1}^m \xi_k B_{ r_k }  \right) \\
 &= \Gamma^m_{t_i+t_j}  K_{1,m} ( 1,t_i),
\end{align*}
with $\Gamma_t := \int_{-t}^t \gamma_0(s)ds$   and
  $K_{1,m} ( 1,t_i)$  defined as  in \eqref{K1}.  The estimate \eqref{F1} yields
\[
 \E\Big[\big\langle A^{\e, a, B^i}_{t_i,x_i}, A^{\e, a, B^j}_{t_j, x_j}   \big\rangle_\H^m \Big] 
  \leq  \frac 1 {2\kappa_0} m!(  8\kappa_0C_N\Gamma_{t_i+t_j})^{m} e^{\frac{t_i D_N}{2C_N} }.
\]
  In particular, for any $n\in\mathbb{N}$, we have 
\begin{align*}
&\quad  \E\Big[\big\vert \langle A^{\e, a, B^i}_{t_i,x_i}, A^{\e, a, B^j}_{t_j, x_j}   \rangle_\H\big\vert^n \Big] \leq 
(2\kappa_0)^{-1/2} \sqrt{  (2n)!  e^{\frac{t_i D_N}{2C_N }}}   (  8\kappa_0C_N\Gamma_{t_i+t_j})^{n}.
\end{align*}
Note that for any given $\lambda\in(0,\infty)$ we can choose    $N>0$ such that $0< 16\kappa_0C_N\Gamma_{t_i+t_j}\lambda < 1$.
Thus, we deduce that
\begin{align*}
 \E\left[ \exp\left( \lambda  \Big\vert  \big\langle A^{\e, a, B^i}_{t_i,x_i}, A^{\e, a, B^j}_{t_j, x_j}   \big\rangle_\H\Big\vert  \right) \right] &\leq 1+ \sum_{n=1}^\infty \frac{\lambda^n}{n!}   \E\Big[\big\vert \langle A^{\e, a, B^i}_{t_i,x_i}, A^{\e, a, B^j}_{t_j, x_j}   \rangle_\H\big\vert^n \Big] \\
 &\leq 1+ (2\kappa_0)^{-1/2}    \sqrt{  e^{\frac{t_i D_N}{2C_N } }} \sum_{n= 1}^\infty \frac{ \sqrt{(2n)! }}{n!} (  8\lambda\kappa_0C_N'\Gamma_{t_i+t_j})^{n} \\
 &\leq 1+(2\kappa_0)^{-1/2}    \sqrt{  e^{\frac{t_i D_N}{2C_N } }} \sum_{n=1}^\infty  (  16\lambda\kappa_0C_N'\Gamma_{t_i+t_j})^{n}, 
 \end{align*}
where we used  $(2n)! \leq 4^n (n!)^2$. Hence, we have proved \eqref{expm} in case {\bf (H1)}.
 
 \bigskip
 Now let us consider case {\bf (H2)} where $\gamma_0(t) = | t|^{2H_0-2}$ for some $H_0\in(1/2,1)$ and  $\varphi(x)= |x|^{1-2H_1}$ for some $H_1\in (0, 1/2)$ such that $H_0 + H_1 > 3/4$. 
We begin with \eqref{begin} and make the change of variables   $t_i-r_k \to r_k$ and $t_j - \wt{r}_k \to \wt{r}_k$, to write, with   $m=2n\geq 2$,  
\begin{align*}
 \E\Big[\big\langle A^{\e, a, B^i}_{t_i,x_i}, A^{\e, a, B^j}_{t_j, x_j}   \big\rangle_\H^m \Big]    
& \leq    \frac{1}{a^{2m}}\int_{[0,a]^{2m}} d\pmb{s_m} d\pmb{\wt{s}_m}\int_{\R^m} \mu(d\pmb{\xi_m})  \int_{[0,t_i]^m\times [0, t_j]^m} d\pmb{r_m} d\pmb{\wt{r}_m}   \\
&  \quad \times
\1_{\{ s_k \leq t_i-r_k , \wt{s}_k\leq t_j-\wt{r}_k ; \forall k\}}         \left( \prod_{k=1}^m | r_k-\wt{r}_k| ^{2H_0-2}\right) \\
& \quad \times   \exp\left(- \Var  \sum_{k=1}^m\xi_k B_{t_i- r_k -s_k} \right)  \exp \left(- \Var \sum_{k=1}^mB_{t_j- \wt{r_k}-\wt{s}_k}) \right). 
\end{align*}
Then, the embedding property \eqref{F15}, together with Cauchy-Schwarz inequality and the change of variables
$  v_k = t_i - r_k - s_k$, $  \wt{v}_k = t_j - \wt{r}_k - \wt{s}_k$, leads to
\begin{align*}
\E\Big[\big\langle A^{\e, a, B^i}_{t_i,x_i}, A^{\e, a, B^j}_{t_j, x_j}   \big\rangle_\H^m \Big]    
& \leq   C_{H_0}^m
 \left(\int_{[0,t_i]^m} d\pmb{v_m}  \left(   \int_{\R^m} \mu(d\pmb{\xi_m})  \exp\left( -  {\rm Var}  \sum_{k=1}^m\xi_k  B^i_{v_k}   \right)   \right)^{\frac{1}{2H_0}} \right)^{H_0} \\
& \quad \times   \left(\int_{[0,t_j]^m} d\pmb{v_m}  \left(   \int_{\R^m} \mu(d\pmb{\xi_m})  \exp\left( -  {\rm Var}  \sum_{k=1}^m\xi_k  B^i_{v_k}   \right)   \right)^{\frac{1}{2H_0}} \right)^{H_0}\\
&= C_{H_0}^m\left[ K_{2,m} (1,t_i) K_{2,m} (1,t_j) \right]^{H_0}.
\end{align*}
Finally,  using the estimate \eqref{F8}, we can write
\[
\E\Big[\big\langle A^{\e, a, B^i}_{t_i,x_i}, A^{\e, a, B^j}_{t_j, x_j}   \big\rangle_\H^m \Big]   
\le  C_{H_0}^m  C_2^{2H_0m}  \big(t_it_j \big)^{m\mathfrak{h}}     \frac { (m!)^{2H_0}  }{ (\Gamma (m \mathfrak{h} +1))^{2H_0}}
\]
 Thus,    for any $\lambda > 0$,
  \begin{align*}
  \E\Big[ e^{ \lambda  \vert \langle A^{\e, a, B^i}_{t_i,x_i}, A^{\e, a, B^j}_{t_j, x_j}   \rangle_\H \vert }\Big]  
& \le 2+ 2   \sum_{n=1}^\infty \frac{\lambda^{2n}}     {(2n)!} 
 \E\Big[ \vert \langle A^{\e, a, B^i}_{t_i,x_i}, A^{\e, a, B^j}_{t_j, x_j}   \rangle_\H \vert^{2n}\Big] \\
&\le 2+2  \sum_{n=1}^\infty  \left[ \lambda C_{H_0}   C_2^{2H_0} \big( t_it_j \big)^{\mathfrak{h}} \right]^{2n}     \frac { ((2n)!)^{2H_0-1}  }{ (\Gamma (2n \mathfrak{h} +1))^{2H_0}}  <\infty,
\end{align*}
because $2H_0-1 -2H_0 \mathfrak {h} = -H_1 <0$.
This implies   \eqref{expm} in case {\bf (H2)}.
 
  \medskip
  
  \noindent{\bf Step 2}. For  fixed $\e>0$ and $i<j$, as $a\downarrow 0$, 
\[
\frac{1}{a^2} \int_0^{a\wedge r} \int_0^{a\wedge \wt{r}} ds_1 ds_2   e^{-\i \xi (B^i_{r-s_1} - B_{\wt{r}-s_2}^j)}
\]
is uniformly bounded and converges to $e^{-\i \xi (B^i_{r} - B_{\wt{r}}^j)}$, due to the path continuity of Brownian motion.  As $\mu_\e(d \xi) = e^{-\e \xi^2}\mu(d \xi) $ is a finite measure and $\gamma_0$ is locally integrable,  we have 
\begin{align*}
\big\langle A^{\e, a, B^i}_{t_i,x_i}, A^{\e, a, B^j}_{t_j, x_j}   \big\rangle_\H   &   \xrightarrow[a\downarrow 0]{a.s.}  \int_\R  \mu(d \xi) e^{-\e \xi^2-\i \xi(x_i-x_j)}  \int_0^{t_i}   \int_0^{t_j} drd\wt{r} \gamma_0(t_i-r-t_j+\wt{r} ) e^{-\i \xi (B^i_{r} - B_{\wt{r}}^j)} \\
& =: \mathcal{I}^{i,j}_{t_i,t_j,\e}(x_i-x_j).
 \end{align*}
 The above almost sure limit $\mathcal{I}^{i,j}_{t_i,t_j,\e}(x_i- x_j)$ is real for any $x_i,x_j\in\R$ and any $\e>0$, since $\mu$ is symmetric (\emph{i.e.} the spectral density is an even function on $\R$).  In the sequel, we just write $\mathcal{I}^{i,j}_\e(x_i-x_j)$,  $\mathcal{I}^{i,j}(x_i-x_j)$ to mean $\mathcal{I}^{i,j}_{t_i, t_j,\e}(x_i- x_j)$ and $\mathcal{I}^{i,j}_{t_i, t_j}(x_i- x_j)$ respectively.

We will prove $\mathcal{I}^{i,j}_\e(x_i-x_j)$ converges, as $\e\downarrow 0$, in $L^2(\Omega)$ (hence in probability) to some limit, denoted by $\mathcal{I}^{i,j}(x_i-x_j)$. Note   that, by Fatou's lemma and  the estimates in {\bf Step 1}, we can establish  that
 for any $\lambda\in\R$,
 \begin{align*}
  \E\Big[ \exp \left(\lambda  \mathcal{I}^{i,j}_\e(x_i-x_j) \right)\Big]\leq   C_\lambda, 
 \end{align*}
 for all $\e>0$, 
where $C_\lambda$ is a constant that does not depend on $\e$. Now we   rewrite 
 \begin{align*}
 \mathcal{I}^{i,j}_\e(x_i-x_j) = \int_{\R} d\xi \varphi(\xi) e^{-\e \xi^2}\int_0^{t_i}   \int_0^{t_j} drd\wt{r} \gamma_0(r-\wt{r} )  e^{-\i \xi (   B^i_{t_i-r} - B_{t_j-\wt{r}}^j + x_i - x_j )} ,
 \end{align*}
 and  we compute for $\e_1, \e_2>0$,
\begin{align*}
\E \big[ \mathcal{I}_{\e_1}^{i,j}(x_i-x_j)\mathcal{I}_{\e_2}^{i,j}(x_i-x_j)\big] &=  \int_{\R^2}   d\pmb{\xi_2} \varphi(\xi_1)\varphi(\xi_2) e^{-\sum_{k=1}^2\e_k \xi^2_k  }   e^{-\i (x_i-x_j)(\xi_1+\xi_2) }\\
&\qquad \times\int_{[0, t_i]^2 \times [0, t_j]^2  }   d\pmb{r_2} d\pmb{\wt{r}_2}   \gamma_0(r_1-\wt{r}_1 ) \gamma_0(r_2-\wt{r}_2 )\\
& \qquad \times  \E\left[ e^{-\i \sum_{k=1}^2\xi_k (   B^i_{t_i-r_k} - B_{t_j-\wt{r}_k}^j  )} \right].
\end{align*}
  Note that $\E\big[ e^{-\i \sum_{k=1}^2\xi_k (   B^i_{t_i-r_k} - B_{t_j-\wt{r}_k}^j  )} \big] = \exp\big[ -\frac{1}{2} \text{Var}  \sum_{k=1}^2\xi_k (   B^i_{t_i-r_k} - B_{t_j-\wt{r}_k}^j  ) \big] $ is positive, and note also from previous calculations  in both cases \textbf{(H1)} and \textbf{(H2)} that
\begin{align*}
&  \int_{\R^2} d\xi_1 d\xi_2 \varphi(\xi_1)\varphi(\xi_2)  \int_{[0, t_i]^2 \times [0, t_j]^2  }   d\pmb{r_2} d\pmb{\wt{r}_2}  \gamma_0(r_1-\wt{r}_1 ) \gamma_0(r_2-\wt{r}_2 ) \\
& \qquad \times  \exp\left( -\frac{1}{2} \text{Var}  \sum_{k=1}^2\xi_k (   B^i_{t_i-r_k} - B_{t_j-\wt{r}_k}^j  ) \right)    <\infty.
\end{align*}
By the dominated convergence theorem, the limit 
$$
\lim_{\e_1,\e_2\downarrow 0}\E \big[ \mathcal{I}_{\e_1}^{i,j}(x_i-x_j)\mathcal{I}_{\e_2}^{i,j}(x_i-x_j)\big]
$$
exists.  Therefore,  as $\e\downarrow 0$, $\mathcal{I}_{\e}^{i,j}(x_i-x_j)$ converges in $L^2(\Omega)$ to some limit $\mathcal{I}^{i,j}(x_i-x_j)$, which is  \emph{formally} given by 
\[
 \mathcal{I}^{i,j}(x_i- x_j) = \int_{\R} d\xi \varphi(\xi)  \int_0^{t_i}   \int_0^{t_j} drd\wt{r} \gamma_0(r-\wt{r} )  e^{-\i \xi (   B^i_{t_i-r} - B_{t_j-\wt{r}}^j + x_i - x_j )}.
\]
 In addition, it is easy to show that this convergence also takes place in $L^p(\Omega)$ for any $p\geq 1$.

Thus, together with  \eqref{expm}, we deduce  by first passing  $a$ to zero, then $\e$ to zero that
\begin{align*}
 \E\left[ \exp\left( \sum_{1\leq i < j \leq k} \big\langle A^{\e, a, B^i}_{t_i,x_i}, A^{\e, a, B^j}_{t_j, x_j}   \rangle_\H \right) \right]  \to  \E\left[ \exp\left( \sum_{1\leq i < j \leq k} \mathcal{I}^{i,j}(x_i-x_j) \right) \right].
\end{align*}

  \noindent{\bf Step 3}. The last step is to establish the  $L^p(\Omega)$-convergence of $u^{\e, a}(t, x)$ to $u(t, x)$, as $\e, a\downarrow 0$. It is enough to show the $L^2(\Omega)$-convergence in view of the moment bounds from {\bf Step 1}.
  First for $\e_1, \e_2, a_1, a_2>0$, we can write by similar arguments as before, 
  \[
  \E\big[  u^{\e_1, a_1}(t, x)u^{\e_2, a_2}(t, x) \big] = \E\Big[ \exp\big( \langle A^{\e_1,a_1, B^1}_{t,x} , A^{\e_2,a_2, B^2}_{t,x}  \rangle_\H \big) \Big]
  \]
   and 
 \begin{align*}
  \langle A^{\e_1,a_1, B^1}_{t,x} , A^{\e_2,a_2, B^2}_{t,x}  \rangle_\H  &=  \int_{\R} \mu(d\xi) e^{-\frac{\e_1+\e_2}{2}\xi^2} \int_{[0,t]^2} dr d\wt{r} \gamma_0(r-\wt{r}) \frac{1}{a_1a_2}\\
  & \quad \times  \int_0^{a_1\wedge r}\int_0^{a_2\wedge \wt{r}}    ds_1ds_2 e^{-\i  \xi ( B^1_{r-s_1} - B^2_{\wt{r} -s_2}   ) }    \\
  &  \xrightarrow[a.s.]{a_1,a_2\downarrow 0} \int_{\R} \mu(d\xi) e^{-\frac{\e_1+\e_2}{2}\xi^2} \int_{[0,t]^2} dr d\wt{r} \gamma_0(r-\wt{r})  e^{-\i  \xi ( B^1_{r} - B^2_{\wt{r}}   ) } \\
  &  \xrightarrow[in~ L^2(\Omega)]{\e_1,\e_2\downarrow 0}  \mathcal{I}^{1,2}_{t,t}(0).  
 \end{align*}
  By \eqref{expm} again,  we have
  \[
   \E\big[  u^{\e_1, a_1}(t, x)u^{\e_2, a_2}(t, x) \big]   \xrightarrow[then~ \e_1,\e_2\downarrow 0]{a_1,a_2\downarrow 0} \E\Big[ e^{   \mathcal{I}^{1,2}_{t,t}(0)  }\Big],
   \]
   which implies that the limit  
   \[
v(t,x):=   \lim_{\e\downarrow 0}   \lim_{a\downarrow 0}  u^{\e, a}(t, x) 
   \]
  exists in $L^p(\Omega)$ for any $p\geq 1$.  
  
  Now consider a test random variable $F = \exp\big(W(g) - \frac{1}{2}\| g \|^2_\H \big)$ for $g\in C^\infty_c(\R_+\times\R)$ and recall that   random variables of this from are dense in $\D^{1,2}$. For such a $F$, we have 
  \begin{align*}
  \E\big[ F  u^{\e, a}(t, x)  \big] &= \E\Big[  \exp\Big( W( A^{\e, a,B}_{t,x}  + g ) - \frac{1}{2} \| A^{\e, a,B}_{t,x} \|^2_\H - \frac{1}{2} \| g \|^2_\H \Big)     \Big] =  \E\Big[  \exp\big( \langle  A^{\e, a,B}_{t,x}  ,  g  \rangle_\H \big)     \Big] \\
  &= \E\left[  \exp\left( \int_0^t  \langle  \varphi_a(s-\bullet)G(\e, B_{t-s}+x - \ast)   ,  g(\bullet,\ast)  \rangle_\H ds \right) \right],
  \end{align*}
 then putting $S_{t,x} =  \E\big[ F u^{\e, a}(t, x)  \big] $, we deduce from the classical Feynman-Kac formula that $S_{t,x}$ solves
 the partial differential equation
 \[
 \partial_t S_{t,x} = \frac{1}{2} \Delta S_{t,x} + S_{t,x}   \langle\varphi_a(t-\bullet)G_{\e}( x - \ast)   ,  g(\bullet,\ast)  \rangle_\H
 \]
 with initial condition $S_{0,x} = \E[F]=1$;   see for instance \cite[page 315]{HN09}. It follows that
 \begin{align}
   \E\big[ F  u^{\e, a}(t, x)  \big]  \notag  &=  1 +  \int_0^t \int_\R \E [ F  u^{\e, a}(s, y)   ] G_{t-s}( x-y)  \langle\varphi_a(s-\bullet)G_{\e}( y - \ast)   ,  g(\bullet,\ast)  \rangle_\H dsdy \notag  \\
\qquad \qquad   \xrightarrow[then~ \e\downarrow 0]{a\downarrow 0}&\ 1 + \E \big\langle DF,     v  G_{t-\bullet}( x-\ast)   \big\rangle_\H,  \label{ast}
   \end{align}
where the convergence in \eqref{ast} is verified at the end of this proof. Assuming \eqref{ast}, we have 
\[
 \E\big[ F  v(t, x)  \big]  = 1 + \E \big\langle DF,     v  G_{t-\bullet}( x-\ast)   \big\rangle_\H,
 \]
    which is equivalent to say that $v(t,x)$ solves the same equation for $u(t,x)$ so that $v(t,x) = u(t,x)$ by the uniqueness of the mild solution; see also  \cite{HN09} for similar arguments for the case where the Gaussian noise is fractional in time and white in space.

 \medskip
 
 Now, let us now justify the convergence in \eqref{ast}. First, by $L^2$-convergence of $u^{\e,a}(s,y)$, we have 
 \[
 \E [ F  u^{\e, a}(s, y)   ]   \xrightarrow[then~ \e\downarrow 0]{a\downarrow 0}  \E[ F v(s,y)].
 \]
   And recall that  $g\in C^\infty_c(\R_+\times\R)$ (suppose $g(r, y) =0$ for $(r,y)\in [0, K]^c\times [-K, K]^c$), so that 
\[
\vert \F g(r, \xi)\vert \leq C \wedge \frac{C}{\xi^2},~ \forall (r,\xi)\in\R_+\times\R.
\]
Combining  the above bound and Dalang's condition, we have
\begin{align}
 \int_{\R} d\xi \varphi(\xi)    \vert \F g\vert(\wt{r}, \xi)  \leq C \int_{\{ \vert \xi \vert \leq 1\}} \varphi(\xi) d\xi + C    \int_{\{ \vert \xi \vert > 1\}} \frac{\varphi(\xi)}{\xi^2} d\xi <+\infty. \label{fact7}
 \end{align}
 Noting also that $\gamma_0$ is locally integrable, we observe that 
 \[
  \langle\varphi_a(s-\bullet)G_{\e}( y - \ast)   ,  g \rangle_\H =\int_{\R_+^2\times\R} drd\wt{r} d\xi \gamma_0(r-\wt{r} ) \varphi_a(s-r)   \varphi(\xi)  e^{-\i y\xi -\frac{\e}{2} \xi^2} \F g(\wt{r}, -\xi) 
 \]
  is uniformly bounded over $(s,y,\e, a)\in[0,t] \times\R\times (0,\infty)^2$. Hence,
 \begin{align}
  \langle\varphi_a(s-\bullet)G_{\e}( y - \ast)   ,  g \rangle_\H &=\ \int_{0}^Kd\wt{r}  (\varphi_a\ast\gamma_0)(s-\wt{r} )  \int_{\R} d\xi \varphi(\xi)  e^{-\i y\xi-\frac{\e}{2} \xi^2}  \F g(\wt{r}, -\xi)     \notag  \notag \\
  &  \xrightarrow{a\to 0} \int_{0}^Kd\wt{r} \gamma_0(s-\wt{r} )  \int_{\R} d\xi \varphi(\xi)  e^{-\i y\xi-\frac{\e}{2} \xi^2}  \F g(\wt{r}, -\xi) \label{DCThere} \\
   & \xrightarrow{\e\to 0}  \int_{0}^Kd\wt{r} \gamma_0(s-\wt{r} )  \int_{\R} d\xi \varphi(\xi)  e^{-\i y\xi}  \F g(\wt{r}, -\xi)  \label{convhere}
  \end{align}
    where we used \eqref{fact7} and the fact that $\varphi_a\ast\gamma_0$ converges in $L^1_{loc}(\R)$ to $\gamma_0$   to obtain \eqref{DCThere} and we applied the dominated convergence theorem in \eqref{convhere}. 
    
 Another application of dominated convergence together with Fubini's theorem leads to 
   \begin{align*}
  \E\big[ F  v(t, x)  \big]  
  =&\ 1 + \int_0^t \int_\R \E [ F  v(s, y)   ] G_{t-s}( x-y)  \int_{\R} d\wt{r} \gamma_0(s-\wt{r} )  \int_{\R} d\xi \varphi(\xi)  e^{-\i y\xi}  \F g(\wt{r}, -\xi)    dsdy \\
  =&\ 1  + \E F \int_0^t ds \int_{\R_+} d\wt{r} \gamma_0(s-\wt{r})   \int_{\R}d\xi  \varphi(\xi)  \F \big(v G_{t-\bullet}( x-\ast)\big)(s, \xi)  \F g(\wt{r}, -\xi)    \\
  =&\  1 + \E F \langle g, vG_{t-\bullet}( x-\ast) \rangle_\H = 1 +  \E  \langle DF, vG_{t-\bullet}( x-\ast) \rangle_\H, ~ (\text{as}~DF = F g),
     \end{align*}
 which confirms the convergence in \eqref{ast}.

 \medskip
 
Therefore,  we can conclude our proof by combining the above steps.   \hfill $\square$


\begin{thebibliography}{999}


\bibitem{BJQ} R.   Balan, M. Jolis and L. Quer-Sardanyons: SPDEs with affine multiplicative fractional noise in space with index $\frac 14 <H<\frac 12$. 
 {\it Electron. J. Probab.} {\bf 20} (2015), no. 54, 36 pp.

 
 \bibitem{BNZ20}
R. Bola\~nos-Guerrero, D. Nualart and G. Zheng:  Averaging 2D Stochastic wave equation. \emph{arXiv:2003.10346}, 2020.
 
 \bibitem{CNN}
 S. Campese, I. Nourdin and D. Nualart: Continuous Breuer-Major theorem: tightness and non-stationarity. \textit{Ann. Probab.} 2020, Vol. 48, No. 1, 147-177.


\bibitem{CKNP19}

L. Chen, D. Khoshnevisan, D. Nualart  and F. Pu (2019). Poincar\'e inequality and central limit theorems for parabolic stochastic partial differential equations. \emph{arXiv preprint}: 1912.01482.

\bibitem{CKNP20}
L. Chen, D. Khoshnevisan, D. Nualart  and F. Pu (2020). Spatial ergodicity and central limit theorem for parabolic Anderson model with delta initial condition. \emph{arXiv preprint}: 2005.10417.




 \bibitem{Dalang99} 
 R. C. Dalang: Extending the Martingale Measure Stochastic Integral With Applications to Spatially Homogeneous S.P.D.E.'s.   \emph{ Electron. J. Probab.} Volume 4 (1999), paper no. 6, 29 pp.

 \bibitem{DNZ18}
 F. Delgado-Vences, D. Nualart and G. Zheng:
 A central limit theorem for the stochastic wave equation with fractional noise.   To appear in:  \textit{Ann. Inst. Henri Poincar\'e Probab. Stat.} (2020+)

\bibitem{GL20}

Y. Gu and J. Li. Fluctuations of a nonlinear SHE in dimensions three and higher.   To appear in: \emph{SIAM Journal on Mathematical Analysis}, 2020.

 

\bibitem{HN05}
Y. Hu and D. Nualart: Renormalized self-intersection local time for fractional Brownian motion. \textit{Ann. Probab.} 33 (2005), 948-983.



 \bibitem{HN09}
  Y. Hu and D. Nualart: Stochastic heat equation driven by fractional noise and local time. \textit{Probab. Theory Relat. Fields} 143 (2009)  : 285-328


\bibitem{HHNT15}
Y. Hu, J. Huang, D. Nualart and S. Tindel: Stochastic heat equations with general multiplicative Gaussian noises: H\"older continuity and intermittency. \emph{Electron. J. Probab.} (2015) {\bf 20}, no. 55, 1-50


\bibitem{HHLNT1}  Y. Hu, J. Huang, K. L\^e, D. Nualart and S. Tindel:  Stochastic heat equation    with rough dependence in space.
  \textit{Ann. Probab.}  {\bf 45} (2017), 4561--4616.
  
  \bibitem{HHLNT2} Y. Hu, J. Huang,  K. L\^e, D. Nualart and S. Tindel:   Parabolic Anderson model with rough dependence on space. In: {\it Computation and Combinatorics in Dynamics, Stochastics and Control.}  Abelsymposium 2016, Celledoni E., Di Nunno G., Ebrahimi-Fard K., Munthe-Kaas H. eds., Springer 2018, 477--498.
  
\bibitem{HLN16}
J. Huang, K. L\^e and D. Nualart:  Large time asymptotics for the parabolic Anderson model driven by spatially  correlated noise. \textit{Ann. Inst. Henri Poincar\'e Probab. Stat.} (2017) 53,  1305-1340.

\bibitem{HLN17}

J. Huang, K. L\^e and D. Nualart:  Large time asymptotics for the parabolic Anderson model driven by space and time correlated noise. \textit{Stoch. PDE: Anal. Comp.} (2017) 5,  614-651.

 

\bibitem{HNV18}
J. Huang, D. Nualart and L.  Viitasaari:  A central limit theorem for the stochastic heat equation. \emph{ArXiv preprint: 1810.09492}, 2018.


\bibitem{HNVZ19}

J. Huang, D. Nualart, L.  Viitasaari and G. Zheng:   Gaussian fluctuations for the stochastic heat equation with colored noise.  \emph{Stoch PDE: Anal Comp} (2020) 8: 402--421
 \href{https://doi.org/10.1007/s40072-019-00149-3}{https://doi.org/10.1007/s40072-019-00149-3}
 
 \bibitem{KNP20}
 D.  Khoshnevisan, D. Nualart and F. Pu(2020). Spatial stationarity, ergodicity and CLT for parabolic Anderson model with delta initial condition in dimension $d\geq 1$
\emph{arXiv preprint}: 2007.01987.



 
\bibitem{MMV01} J. M\'emin, Y. Mishura and E. Valkeila: Inequalities for the moments of Wiener integrals with respect to a fractional Brownian motion. \textit{Stat. Probab. Lett.} Volume 51, Issue 2,  2001, Pages 197-206

 \bibitem{bluebook}

 I.  Nourdin and  G.  Peccati (2012).   \textit{Normal approximations with Malliavin calculus: from Stein's method to universality}, Cambridge tracts in Mathematics, Vol. 192, Cambridge University Press.

 
 
\bibitem{Nualart} 
D. Nualart: The Malliavin calculus and related topics.  Second edition. Probability and its Applications (New York). \textit{Springer-Verlag, Berlin,} 2006. xiv+382 pp. 






\bibitem{NP05}
D. Nualart and G. Peccati:  Central limit theorems for sequences of multiple stochastic  integrals.  \textit{Ann. Probab.}, 33(1):177--193, 2005.

 \bibitem{NZ19BM}
 
 D. Nualart and G. Zheng: Averaging Gaussian functionals. \textit{Electron. J. Probab.} 25 (2020), no. 48, 1-54. \href{https://doi.org/10.1214/20-EJP453}{https://doi.org/10.1214/20-EJP453}
 
 \bibitem{NZ20a}
 
 D. Nualart and G. Zheng (2020).  Central limit theorems for stochastic wave equations in dimensions one and two. \emph{arXiv preprint}: 2005.13587
 
 \bibitem{PT05}
  G. Peccati and C.A. Tudor.  Gaussian limits for vector-valued multiple stochastic integrals,
\emph{S\'eminaire de Probabilit\'es XXXVIII}, 247-262, Lecture Notes in Math., 1857, Springer, Berlin,
2005. MR-2126978


\bibitem{PU20} 
 F. Pu (2020). Gaussian fluctuation for spatial average of parabolic Anderson model with Neumann/Dirichlet boundary conditions. \emph{arXiv preprint}:2008.08267
 
 \bibitem{SSX19}
 
J. Song, X. Song and F. Xu:   Fractional stochastic wave equation driven by a Gaussian noise rough in space. \emph{Bernoulli} Volume 26, Number 4 (2020), 2699-2726.
\href{https://doi.org/10.3150/20-BEJ1204}{https://doi.org/10.3150/20-BEJ1204}


\end{thebibliography}
\end{document}